\documentclass[a4paper,reqno]{amsart}
\usepackage[english]{babel}

\usepackage{amsmath,amstext,amssymb,amsthm}
\usepackage{mathrsfs}
\usepackage{bbm}
\usepackage{enumitem}

\usepackage{hyperref}

\newcommand{\arXivlink}[1]{\href{https://arxiv.org/abs/#1}{\texttt{arXiv:#1}}}

\newcommand*\RR{\mathbb{R}}
\newcommand*\CC{\mathbb{C}}
\newcommand*\NN{\mathbb{N}}
\newcommand*\ZZ{\mathbb{Z}}

\newcommand*\Lapl{\mathscr{L}}
\newcommand*\Rsl{\mathfrak{R}}

\newcommand*{\Rnon}{\RR_+}
\newcommand*{\Rpos}{\mathring{\RR}_+}

\newcommand*\ind{\mathbbm{1}}

\newcommand{\sobolev}[2]{L^{#2}_{#1}}
\newcommand{\hormander}[1]{\mathcal{H}_2^{#1}}

\DeclareMathOperator{\sign}{sign}

\newcommand*\supp{\mathop{\mathrm{supp}}}
\newcommand*\arccosh{\mathop{\mathrm{arccosh}}}
\newcommand*{\lie}[1]{\mathfrak{#1}}
\newcommand*\id{\mathrm{id}}

\newcommand*{\defeq}{\mathrel{:=}}

\newcommand*{\dist}{\varrho}

\newcommand*\Df{\mathcal{D}}
\newcommand*\LinBnd{\mathcal{L}}

\newcommand*\dd{\mathrm{d}}

\newcommand*\bdx{\mathbf{x}}
\newcommand*\bdy{\mathbf{y}}
\newcommand*\bdz{\mathbf{z}}

\newcommand*\Dom{\mathfrak{D}}

\newcommand*\Riesz{\mathcal{R}}
\newcommand*\iker{\mathcal{K}}

\DeclareMathOperator{\tr}{tr}

\theoremstyle{plain}
\newtheorem{thm}{Theorem}[section]
\newtheorem{lm}[thm]{Lemma}
\newtheorem{prop}[thm]{Proposition}
\newtheorem{cor}[thm]{Corollary}

\theoremstyle{remark}
\newtheorem{rem}[thm]{Remark}

\numberwithin{equation}{section}

\begin{document}
\title[Riesz transforms on solvable extensions of Carnot groups]{$L^p$-boundedness of Riesz transforms on solvable extensions of Carnot groups}

\author[A. Martini]{Alessio Martini}
\address[A. Martini]{Dipartimento di Scienze Matematiche ``G.L. Lagrange'' \\ Politecnico di Torino \\ Corso Duca degli Abruzzi 24 \\ 10129 Torino \\ Italy}
\email{alessio.martini@polito.it}

\author[P. Plewa]{Pawe\l{} Plewa}
\address[P. Plewa]{Dipartimento di Scienze Matematiche ``G.L. Lagrange'',
	Politecnico di Torino\\ Corso Duca degli Abruzzi 24 \\ 10129 Torino \\ Italy
	and Department of Pure and Applied Mathematics, 
	Wroc{\l}aw University of Science and Technology\\ wyb. Wys\-pia{\'n}\-skie\-go  27\\ 50–-370 Wroc{\l}aw\\ Poland }        
\email{pawel.plewa@pwr.edu.pl}

\begin{abstract}
Let $G=N\rtimes \RR$, where $N$ is a Carnot group and $\RR$ acts on $N$ via automorphic dilations.
Homogeneous left-invariant sub-Laplacians on $N$ and $\RR$ can be lifted to $G$, and their sum is a left-invariant sub-Laplacian $\Delta$ on $G$. We prove that the first-order Riesz transforms $X \Delta^{-1/2}$ are bounded on $L^p(G)$ for all $p\in(1,\infty)$, where $X$ is any horizontal left-invariant vector field on $G$. This extends a previous result by Vallarino and the first-named author, who obtained the bound for $p\in(1,2]$. The proof makes use of an operator-valued spectral multiplier theorem, recently proved by the authors, and hinges on estimates for products of modified Bessel functions and their derivatives.
\end{abstract}



\thanks{The authors gratefully acknowledge the financial support of Compagnia di San Paolo. The first-named author is a member of Gruppo Nazionale per l'Analisi Matematica, la Probabilit\`a e le loro Applicazioni (GNAMPA) of Istituto Nazionale di Alta Matematica (INdAM). The second-named author is also supported by the Foundation for Polish Science (START 057.2023).}

\maketitle

\section{Introduction}

Let $N$ be a stratified Lie group of step $r$. In other words, $N$ is a connected, simply connected nilpotent Lie group, whose Lie algebra $\lie{n}$ is decomposed as the direct sum $\lie{n}_1 \oplus \dots \oplus \lie{n}_r$ of subspaces, called layers, which form a gradation of $\lie{n}$ and such that the first layer $\lie{n}_1$ generates $\lie{n}$ as a Lie algebra. Let $\breve{X}_1,\dots,\breve{X}_d$ be left-invariant vector fields forming a basis of the first layer of the Lie algebra of $N$; equipped with the sub-Riemannian structure induced by the vector fields $\breve{X}_1,\dots,\breve{X}_d$, the group $N$ is also known as a Carnot group. Let
\begin{equation*}
\Lapl = -\sum_{j=1}^{d} \breve{X}_j^2
\end{equation*}
be the corresponding sub-Laplacian on $N$. Beside being left-invariant, the sub-Laplacian $\Lapl$ is homogeneous with respect to the automorphic dilations on $N$ induced by the gradation of $\lie{n}$.

We consider the semidirect product group $G=N\rtimes \RR$, where $\RR$ acts on $N$ via automorphic dilations. The group $G$ is solvable, amenable and of exponential growth. We equip $G$ with the right Haar measure. Moreover, we lift the vector fields $\breve{X}_j$ on $N$ to the left-invariant vector fields $X_j$ on $G$; together with the left-invariant vector field $X_0$ in the direction of $\RR$, they generate the Lie algebra of $G$ and define a sub-Riemannian structure on the manifold $G$. The associated left-invariant sub-Laplacian on $G$ is given by
\begin{equation}\label{eq:subLaplacian}
\Delta= -\sum_{j=0}^{d} X_j^2.
\end{equation}
The operator $\Delta$, initially defined on $C^\infty_c(G)$, extends uniquely to a positive self-adjoint operator on $L^2(G)$. 

In this work, we consider the first-order Riesz transforms associated with $\Delta$ on $G$, given by
\begin{equation*}
\Riesz_j = X_j \Delta^{-1/2},\qquad j=0,\ldots,d.
\end{equation*}
In \cite{MaVa} Vallarino and the first-named author proved that the operators $\Riesz_j$ are of weak type $(1,1)$ and bounded on $L^p(G)$ for all $p\in(1,2]$. Here we extend that result to the full range of $p$. 

\begin{thm}\label{thm:main}
The Riesz transforms $\Riesz_j$, $j=0,\ldots,d$, are bounded on $L^p(G)$ for all $p\in(1,\infty)$.
\end{thm}

By linearity, Theorem \ref{thm:main} is equivalent to the $L^p$-boundedness for $p \in (1,\infty)$ of the Riesz transforms $X \Delta^{-1/2}$ for all horizontal left-invariant vector fields $X$ on $G$, i.e., for all $X$ in the linear span of $X_0,\dots,X_d$. Moreover, by duality and composition, Theorem \ref{thm:main} also implies the $L^p$-boundedness for $p\in (1,\infty)$ of the second-order Riesz transforms $\Riesz_k^* \Riesz_j = X_k \Delta^{-1} X_j$ for all $j,k=0,\dots,d$. The latter boundedness result was already proved in \cite{MaVa}, by exploiting an additional cancellation coming from the composition: indeed, as it turns out, the convolution kernel of $\Riesz_k^* \Riesz_j$ is integrable at infinity. However, this integrabiliy property does not hold for the first-order Riesz transforms $\Riesz_j$.

The study of the $L^p$-boundedness of Riesz transforms associated to Laplacians and sub-Laplacians on Lie groups and more general manifolds is a classical problem in harmonic and geometric analysis. 
We shall not attempt to give a complete account of its long history, and we shall limit our discussion to the works that are most relevant to our problem.
Particularly studied is the case where the underlying manifold is doubling (see, e.g., \cite{ACDH,BG,CouDu,tElRoSi,LoVa,Si3}); however, our group $G$ has exponential volume growth, so it does not fall into this framework, and indeed the main challenge in the study of the Riesz transforms $\Riesz_j$ comes from their behaviour at infinity. In these regards, other works in the literature (see, e.g., \cite{ACDH,Ba,Li,LoMu2}) consider ``shifted'' Riesz transforms of the form $X_j (a+\Delta)^{-1/2}$ for some $a>0$, or exploit the spectral gap of the relevant Laplacian, in order to gain some ``extra decay'' at infinity; however, the sub-Laplacian $\Delta$ on $G$ has no spectral gap and the Riesz transforms $\Riesz_j = X_j \Delta^{-1/2}$ are unshifted, so a different approach is needed.

When $N$ is abelian (i.e., $r=1$), then, up to isomorphisms, $N = \RR^d$ and $\Lapl$ is the standard Euclidean Laplacian. In this case, $\Delta$ is in fact a full (elliptic) Laplacian on the $ax+b$ group $G = \RR^d \rtimes \RR$, and Theorem \ref{thm:main} reduces to known results.
When $d=1$ Gaudry and Sj\"ogren \cite{GaSj,Sj'99} proved that $\Riesz_1$ is bounded on $L^p(G)$ for $p \in (1,\infty)$, but left open the analogous problem for $\Riesz_0$. For any dimension $d\geq 1$, Hebisch and Steger \cite{HeSt} proved that each Riesz transform $\Riesz_j$ ($j=0,\dots,d$) is bounded on $L^p(G)$ in the restricted range $p\in(1,2]$. The case $p>2$ has been fully settled only recently by the first-named author \cite{Ma23}.
The fact that the $L^p$-boundedness of Riesz transforms for $p>2$ is a more delicate question than that for $p<2$ is not surprising and has also been noticed in other settings (see, e.g., \cite{ACDH,CouDu}).

One of the crucial contributions in \cite{HeSt} is the development of a Calder\'on--Zygmund theory adapted to the nondoubling geometry of the group $G = \RR^d \rtimes \RR$. Once this theory is established, the $L^p$-boundedness for $p \in (1,2]$ of the Riesz transforms $\Riesz_j$ can be reduced via relatively standard arguments to suitable uniform-in-time weighted $L^1$-bounds for the heat kernel and its gradient; in the case of $G = \RR^d \rtimes \RR$, the latter bounds are proved in \cite{HeSt} by exploiting the explicit formulas that are available for the heat kernel in that setting. The approach of \cite{HeSt} also yields a weak type $(1,1)$ endpoint bound for the $\Riesz_j$. An additional endpoint bound is provided by the Hardy space theory on $G$ developed in \cite{Va-PhD,Va}, yielding the $H^1(G) \to L^1(G)$ boundedness of the singular integral operators encompassed by the theory of \cite{HeSt}, including the Riesz transforms $\Riesz_j.$

Interestingly enough, the adjoint Riesz transforms $\Riesz_j^*$ are \emph{not} bounded from $H^1(G)$ to $L^1(G)$; see \cite[Theorems~4.2 and 5.2]{SjVa} and \cite[Proposition~7.1]{Ma23}. This shows that the $p>2$ bounds for the $\Riesz_j$ cannot be simply obtained via duality considerations by means of the same approach used for $p<2$, and a different approach is needed. In \cite{Ma23}, the proof of Riesz transform bounds for $p>2$ is crucially based on the observation that $L^p(\RR^d \rtimes \RR) \cong L^p(\RR^d;L^p(\RR))$, and moreover left-invariant operators on $G = \RR^d \rtimes \RR$ are also translation-invariant with respect to $\RR^d$-translations; thus, such operators can be thought of as operator-valued Fourier multipliers on $\RR^d$, where their operator-valued symbols act on functions on $\RR$, and their $L^p$-boundedness properties can be obtained by means of the operator-valued Fourier multiplier theorem of \cite{StWe,We}. In order to apply the operator-valued multiplier theorem to the Riesz transforms $\Riesz_j$, precise asymptotics for the convolution kernels of the $\Riesz_j$ at the origin and at infinity are proved in \cite{Ma23}, which again make fundamental use of explicit formulas for the heat kernel.

The lack of comparatively explicit formulas and asymptotics in the case of $G = N \rtimes \RR$ for nonabelian $N$ is one of the main issues that must be faced in extending the aforementioned results, thus forcing one to develop more ``conceptual'' approaches. This is already true in \cite{MOV,MaVa}: in those works, the Calder\'on--Zygmund and Hardy space theory of \cite{HeSt,Va-PhD} is extended to the case of $G=N \rtimes \RR$ for an arbitrary stratified Lie group $N$, and moreover uniform-in-time gradient heat kernel bounds for the sub-Laplacian $\Delta$ are proved, which yield the $L^p$-boundedness for $p \in (1,2]$ of the Riesz transforms $\Riesz_j$. Among other things, as no explicit formulas for the heat kernel on $G$ are available, the proof of the heat kernel bounds in \cite{MOV,MaVa} requires a different method, exploiting the relation between the heat kernels on $G$ and $N$ \cite{G}.

Beside the lack of explicit formulas, an additional ``methodological'' obstacle appears when trying to extend to $G =N \rtimes \RR$ for nonabelian $N$ the approach for the $p>2$ bounds from \cite{Ma23}. Indeed, in the case of an arbitrary stratified Lie group $N$, it is still true that $L^p(G) \cong L^p(N;L^p(\RR))$, and that left-invariant operators on $G$ are also left-invariant with respect to $N$-translations.
However, an operator-valued Fourier multiplier theorem on an arbitrary stratified group $N$, analogous to that of \cite{StWe,We} for $\RR^d$, is not available in the literature. As a matter of fact, it is not even clear what such an analogue would be exactly, given the complicated nature of the group Fourier transform and $L^p$ Fourier multiplier theorems on a nonabelian stratified group $N$, cf.\ \cite{DeM,FR,Lin}.

A similar problem was faced in \cite{MaPl2}, where we studied the $L^p$-boundedness of Riesz transforms on ``$ax+b$ hypergroups'' (an analogue of the $ax+b$ groups $\RR^d \rtimes \RR$ in the case where $d$ is fractional, see also \cite[Section 4]{MaPl}). There we proved an operator-valued spectral multiplier theorem, which can be used as a replacement of \cite{StWe,We}, and it is that result that we shall apply here as well.
The operator-valued multiplier theorem in \cite{MaPl2} is a general result, which, in a nutshell, shows the existence of an operator-valued functional calculus for any self-adjoint operator with a scalar-valued H\"ormander functional calculus on $L^p$; the latter property is satisfied by the sub-Laplacian $\Lapl$ on $N$, thanks to the classical result by Christ \cite{Ch} and Mauceri and Meda \cite{MaMe}.

To see why the operator-valued functional calculus for $\Lapl$ is relevant to the Riesz transforms on $G$, observe that, in suitable coordinates on $G = N \rtimes \RR$,
\begin{gather}
\label{eq:vfsG} X_0 = \partial_u, \quad X_j = e^u \breve{X}_j \ (j=1,\dots,d), \\
\label{eq:DeltaLapl} \Delta = -\partial_u^2+e^{2u} \Lapl,
\end{gather}
where $u$ denotes the coordinate along $\RR$. Thus, at least formally, we can write
\[
\Riesz_0 = X_0 \Delta^{-1/2} = F_0(\sqrt{\Lapl}), \qquad F_0(\xi) \defeq \partial_u H(\xi)^{-1/2},
\]
where $H(\xi) \defeq -\partial_u^2+\xi^2 e^{2u}$ is the Schr\"odinger operator with potential $\xi^2 e^{2u}$ on $\RR$ for any $\xi >0$; in a similar way, for $j>0$,
\[
\Riesz_j = X_j \Delta^{-1/2} = \Riesz_j^N F_1(\sqrt{\Lapl}), \qquad F_1(\xi) \defeq \xi e^u H(\xi)^{-1/2},
\]
where the $\Riesz_j^N \defeq \breve{X}_j \Lapl^{-1/2}$ are the Riesz transforms on $N$, which are known to be $L^p$-bounded for $p \in (1,\infty)$ \cite{ChGe,tElRoSi,Fo75,LoVa}. Thus, the $L^p$-boundedness of the Riesz transforms $\Riesz_j$ on $G$ can be reduced to that of the operator-valued functions $F_j(\sqrt{\Lapl})$ of the homogeneous sub-Laplacian $\Lapl$ on $N$.

Notice that the $F_j(\xi)$ are nothing else than the first-order Riesz transforms associated with the Schr\"odinger operator $H(\xi)$ on $\RR$. The operator-valued multiplier theorem of \cite{MaPl2} reduces the $L^p(G)$-boundedness of $F_j(\sqrt{\Lapl})$ for $p \in (1,\infty)$ to checking that the symbols $F_j(\xi)$ satisfy assumptions of Mihlin type of the form
\begin{equation}\label{eq:wL2bd}
\sup_{\xi > 0} \|\xi^n \partial_\xi^n F_j(\xi)\|_{L^2(w) \to L^2(w)} \lesssim_{n,[w]_2} 1
\end{equation}
for all $n \in \NN$ and all weights $w$ in the Muckenhoupt class $A_2(\RR)$, where $[w]_2$ is the $A_2$-characteristic of the weight. By Rubio de Francia's extrapolation theorem, for any $p \in (1,\infty)$ the bound \eqref{eq:wL2bd} is equivalent to the bound
\[
\sup_{\xi > 0} \|\xi^n \partial_\xi^n F_j(\xi)\|_{L^p(w) \to L^p(w)} \lesssim_{n,[w]_p} 1
\]
for all $n \in \NN$ and $w \in A_p(\RR)$; in particular, by taking $n = 0$ and $w \equiv 1$ in the latter bound, we see that the condition \eqref{eq:wL2bd} can be thought of as a strengthening of the $L^p(\RR)$-boundedness of the Riesz transforms $F_j(\xi)$ associated with $H(\xi)$.

Interestingly enough, the (unweighted) $L^p(\RR)$-boundedness of the $F_j(\xi)$ for all $p \in (1,\infty)$ was proved in \cite{Ma23}, by means of transference, as a consequence of the $L^p(G)$-boundedness of the $\Riesz_j$ on $G = \RR^d \rtimes \RR$; notice that the new result there is again for $p>2$, as the $L^p$-boundedness for $p \in (1,2]$ of first-order Riesz transforms associated with Schr\"odinger operator with nonnegative potentials is known in great generality \cite{Si3}. Here the direction is somewhat reversed, in the sense that we deduce the $L^p(G)$-boundedness of the $\Riesz_j$ from (a substantially strengthened version of) the $L^p(\RR)$-boundedness of the $F_j(\xi)$.

The crucial advantage of the approach presented here is that it applies to any $G = N \rtimes \RR$ for an arbitrary stratified group $N$; as a matter of fact, this approach would appear to have the potential for extensions to other settings, where the relevant Laplacian $\Delta$ has the form \eqref{eq:DeltaLapl} for a suitable operator $\Lapl$ (cf.\ the setup of \cite{G}).
Indeed, differently from \cite{Ma23}, here we do not need any explicit formula or asymptotic expression for the convolution kernels of the $\Riesz_j$ on $G$, and instead we rely on the functional calculus for the Schr\"odinger operators $H(\xi)$ on $\RR$. 
Specifically, by means of the subordination formula
\begin{equation}\label{eq:subHxi}
H(\xi)^{-1/2} = \frac{2}{\pi} \int_0^\infty (t^2+H(\xi))^{-1} \,dt,
\end{equation}
the study of the $F_j(\xi)$ is reduced to that of the resolvent operators $(t^2+H(\xi))^{-1}$; moreover, explicit formulas for the integral kernels of these resolvents is available, in terms of products of modified Bessel functions \cite{Hu,Ti}:
\[
(t^2+H(\xi))^{-1} f(u) = \int_\RR I_t(\xi e^{\min(u,v)}) K_t(\xi e^{\max(u,v)}) f(v) \,dv.
\]
Thus, our result reduces eventually to certain pointwise estimates for products of modified Bessel functions and derivatives thereof.

For technical reasons, in the proof below we actually split each of the Riesz transforms $\Riesz_j$ into a ``local part'' and a ``part at infinity'', corresponding to an appropriate splitting of the subordination integral \eqref{eq:subHxi}. The local part can be treated in a more traditional way, by means of heat kernel estimates on $G$, and it is the part at infinity that shall be tackled by means of the operator-valued multiplier theorem. Thus, we shall actually check the conditions \eqref{eq:wL2bd} for certain variants of the symbols $F_j(\xi)$, which only involve Bessel functions $I_t$ and $K_t$ of small order $t$.

Our approach, by its nature, does not provide endpoint bounds for the Riesz transforms. At the $p=1$ endpoint, we already know from \cite{MaVa} that the $\Riesz_j$ are of weak type $(1,1)$ and bounded from $H^1(G)$ to $L^1(G)$; see also the recent work \cite{SjVa2}, where a sharper $H^1(G) \to H^1(G)$ bound is proved for $\Riesz_1,\Riesz_2$ in the case of $N = \RR^2$.
At the other endpoint, we know that the $H^1(G) \to L^1(G)$ boundedness fails for the adjoint transforms $\Riesz_j^*$ when $N$ is abelian \cite{Ma23,SjVa}, and even in the hypergroup setting \cite{MaPl2} or in the discrete setting of flow trees \cite{LMSTV,MSTV}, so one may expect it to fail for nonabelian $N$ as well. On the other hand, the validity of a weak type $(1,1)$ bound for the $\Riesz_j^*$ appears to be a subtler problem, which is not fully understood even in the case of $N$ abelian: indeed, when $N = \RR^d$, weak type $(1,1)$ bounds for the $\Riesz_j^*$ are known to hold in the case $j>0$ \cite{GaSj,Ma23}, but the question for $j=0$ remains open; moreover, the proofs of the known bounds make fundamental use of explicit formulas and asymptotics for the Riesz transform kernels. So new methods and ideas will likely be needed to tackle this problem for the Riesz transforms on $G = N \rtimes \RR$ for an arbitrary stratified group $N$.

\subsection*{Structure of the paper}
In Section \ref{S:Preliminaries} we recall some basic facts about stratified Lie groups and their semidirect extensions. Among other things, we summarize a few results from the singular integral theory on $G$; moreover, we state some useful estimates for the heat kernels on $G$ and $N$, as well as a formula connecting them.

In Section \ref{S:Riesz_tr} we set up our study of the Riesz transforms on $G$.
Specifically, we decompose each $\Riesz_j$ into a local part and a part at infinity, and prove the $L^p$-boundedness of the local parts by means of heat kernel estimates. We then recall the operator-valued spectral multiplier theorem from \cite{MaPl2}, and show how the $L^p$-boundedness of the parts at infinity of the Riesz transforms reduces to weighted $L^2$-bounds of homogeneous derivatives of certain operator-valued symbols.

The required weighted $L^2$-bounds are finally proved in Section \ref{S:M_j-estimates}, by means of explicit formulas for the integral kernels of the relevant operators, involving modified Bessel functions. Section \ref{S:Appendix} is an appendix in which we collect the estimates for Bessel functions that we need in the proof.

\subsection*{Notation}
The symbols $\RR$, $\ZZ$, and $\CC$ have standard meanings. For the set of all positive and nonnegative integers we write $\NN_+$ and $\NN=\NN_+\cup\{0\}$, respectively. The closed and open half-lines are denoted by $\Rnon=[0,\infty)$ and $\Rpos=(0,\infty)$.

We write $\LinBnd(V,W)$ for the space of bounded linear operators between two topological vector spaces $V$ and $W$, and abbreviate $\LinBnd(V,V)$ as $\LinBnd(V)$. Moreover, we use the notation $\iker_T$ for the integral kernel of an integral operator $T$.

For nonnegative quantities $X,Y$ we write $X\lesssim Y$ if there exists $c\in\Rpos$ such that $X\leq c Y$. We write $X\simeq Y$ if $X\lesssim Y$ and $Y\lesssim X$ simultaneously. If we want to emphasize that the implicit constant $c$ depends on a parameter $a$ we use subscripted variants such as $\lesssim_a$ and $\simeq_a$.

\section{Preliminaries}\label{S:Preliminaries}

In this section we recall a number of important properties of the group $G$ and the sub-Laplacian $\Delta$; we refer to \cite{MOV,MaVa} for a more extensive discussion.

\subsection{Stratified groups and their solvable extension}

As in the introduction, let $N$ be a stratified Lie group of step $r\in\NN_+$. Thus $N$ is a simply connected, connected Lie group, whose Lie algebra $\lie{n}$ is equipped with a derivation $D$ such that the eigenspace of $D$ corresponding to the eigenvalue $1$ generates $\lie{n}$ as a Lie algebra. The set of all eigenvalues of $D$ is $\{1,\ldots,r\}$, and the eigenspace associated to the eigenvalue $j$ is called $j$th layer and denoted by $\lie{n}_j$ for $j=1,\ldots,r$. Moreover, $\lie{n}$ is the direct sum of the layers $\lie{n}_j$, and $N$ is nilpotent.

We identify $N$ with $\lie{n}$ through the exponential map; via this identification, Lebesgue measure on $\lie{n}$ is a (left and right) Haar measure on $N$.
We denote by $(\delta_t)_{t>0}$ the family of automorphic dilations on $N$ given by $\delta_t=\exp((\log t)D)$.
Notice that $\det \delta_t = t^Q$, where $Q=\tr D= \dim\lie{n}_1+2\dim\lie{n}_2+\ldots+r \dim\lie{n}_r$ is called the homogeneous dimension of $N$.

We define the semidirect product $G=N\rtimes\RR$, where the additive group $\RR$ acts on $N$ via $u \mapsto \delta_{e^u}$. Usually we will write the elements of $G$ as pairs $(x,u)$ where $x\in N$ and $u\in \RR$, although sometimes in order to make the formulas more transparent we will use the bold notation $\bdx\in G$. The product law and the inversion map on $G$ are given by
\begin{equation*}
(x,u) \cdot (x',u') = \big(x\cdot \delta_{e^u}(z'),u+u'\big),\qquad (x,u)^{-1} = \big( -\delta_{e^{-u}}(z),-u\big),
\end{equation*}
and the identity element is $0_G = (0_N,0)$.

The group $G$ is a solvable Lie group. Unlike $N$, the group $G$ is not unimodular. The right and left Haar measure are
\begin{equation*}
\dd x\, \dd u\qquad \text{and} \qquad m(x,u)\, \dd x\,\dd u
\end{equation*}
respectively, where $m(x,u)=e^{-Qu}$ is the modular function. Notably, $G$ has exponential volume growth. Throughout the paper we will always use the right Haar measure to define the Lebegue spaces $L^p(G)$, $p\in [1,\infty)$.

Let us choose left-invariant vector fields $\breve{X}_1,\dots,\breve{X}_d$ on $N$ that form a basis of the first layer $\lie{n}_1$, and define the left-invariant vector fields $X_0,\dots,X_d$ on $G$ as in \eqref{eq:vfsG}. The systems $\{\breve{X}_j\}_{j=1,\dots,d}$ and $\{X_j\}_{j=0,\dots,d}$ are bracket-generating on $N$ and $G$ respectively; let $\dist_N$ and $\dist$ denote the corresponding Carnot--Carath\'eodory distances on $N$ and $G$. We write $|\bdx|_\dist = \dist(\bdx,0_G)$ for the distance of $\bdx\in G$ from the origin; similarly $|x|_{\dist_N}=\dist_N(x,0_N)$ stands for the distance of $x\in N$ from the origin. The following relation between $\dist$ and $\dist_N$ was proved in \cite[Proposition~2.7]{MOV} (see also \cite{He99}):
\begin{equation*}
|(x,u)|_\dist = \arccosh \left( \cosh u + \frac{|x|_{\dist_N}^2}{2e^u}\right),\qquad (x,u)\in G.
\end{equation*}

The convolution on $G$ is given by
\begin{equation*}
f\ast g(\bdx) = \int_G f(\bdx\bdy^{-1}) \, g(\bdy) \,\dd\bdy
\end{equation*}
and the $L^1$-isometric involution by
\begin{equation*}
f^\ast(\bdx) = m(\bdx) \overline{f(\bdx^{-1})}.
\end{equation*}
Of course, the above definitions make sense as written for sufficiently well-behaved functions $f$ and $g$ on $G$, but can be extended to more general classes of functions and distributions on $G$.

Let $\Df'(G)$ denote the space of distributions on $G$. By the Schwartz kernel theorem, any bounded left-invariant operator $T : C^\infty_c(G) \to \Df'(G)$ is a (right) convolution operator, and we shall denote by $K_T \in \Df'(G)$ its convolution kernel, i.e.,
\[
Tf = f * K_T \qquad \forall f \in C^\infty_c(G).
\]
An analogous result holds for left-invariant operators on $N$.

\subsection{Calder\'on--Zygmund theory}

It was proved in \cite{MOV} that $G$ equipped with the distance $\dist$ and the right Haar measure is an abstract Calder\'on--Zygmund space in the sense of Hebisch and Steger \cite{HeSt}, and that moreover an atomic Hardy space $H^1(G)$ adapted to this structure, in the sense of \cite{Va-PhD}, can be defined on $G$. This generalisation of the classical Calder\'on--Zygmund and Hardy space theory allows us to treat singular integrals on the exponentially growing group $G$.

We recall a version of the $L^p$-boundedness theorem for singular integrals proved by Hebisch and Steger \cite[Theorem~1.2]{HeSt} and Vallarino \cite[Theorem 3.10]{Va-PhD}. The statement below (analogous to \cite[Theorem~2.3]{MaVa}) is adapted to convolution operators and also includes the $H^1 \to L^1$ endpoint.
 
\begin{thm}\label{thm:CZ}
Let $T$ be a linear operator bounded on $L^2(G)$. Assume that
\[
\langle T f, g \rangle =\sum_{n\in\ZZ}\langle T_n f , g \rangle
\]
for all compactly supported $f,g \in L^2(G)$ with disjoint supports, where the $T_n$ are convolution operators associated with kernels $K_n$ that satisfy, for certain $c,B,\varepsilon>0$ with $c\neq 1$, the conditions
	\begin{align*}
	\int_{G} \big| K_n(\bdz)\big| \big(1+c^n|\bdz|_\dist \big)^\varepsilon\, \dd \bdz&\leq B\\
	\int_G \big| X_j K_n^\ast (\bdz)\big| \, \dd\bdz &\leq B c^n,\qquad j=0,\ldots,d.
	\end{align*}
	Then $T$ is of weak type $(1,1)$, bounded on $L^p(G)$ for $p\in(1,2]$, and bounded from $H^1(G)$ to $L^1(G)$. 
\end{thm}

\subsection{Heat kernels}

Let $\{e^{-t\Delta} \}_{t>0}$ denote the heat semigroup associated with $\Delta$. For each $t>0$ the operator $e^{-t\Delta}$ is a convolution operator and the heat kernel $K_{e^{-t\Delta}}$ is expressed in terms of the heat kernel $K_{e^{-t\Lapl}}$ associated with the sub-Laplacian $\Lapl$ by 
\begin{equation*}
K_{e^{-t\Delta}}(x,u) = \int_0^\infty \psi_t(\zeta) \exp\left(-\frac{\cosh u}{\zeta}\right) K_{e^{-\frac{e^u \zeta}{2}\Lapl}}(x)\,\dd\zeta,
\end{equation*}
see \cite[\S2]{G}, where
\begin{equation}\label{eq:psi}
\psi_t(\zeta) =\frac{e^{\frac{\pi^2}{4t}}}{\zeta^2\sqrt{4\pi^3t}} \int_0^\infty \sinh\theta \sin\frac{\pi\theta}{2t} \exp\left( -\frac{\theta^2}{4t}- \frac{\cosh\theta}{\zeta}\right)\,\dd\theta.
\end{equation}
Note that $|\psi_t(\zeta)|\lesssim_t \zeta^{-2}$ for $\zeta>0$.

In the next proposition we gather some heat kernel estimates proved in \cite{MaVa}; see also \cite[Proposition~4.2]{MOV}. To state them, we recall the noncommutative multiindex notation
\[
X^\alpha = X_{\alpha_1} \cdots X_{\alpha_N} , \qquad |\alpha| = N,
\]
for $\alpha = (\alpha_1,\dots,\alpha_N) \in \{0,\dots,d\}^N$ and $N \in \NN$.

\begin{prop}[{\cite[Propositions~3.1 and 3.5]{MaVa}}]\label{prop:MaVa}
If $\varepsilon\geq 0$ and $j\in\{0,\ldots,d\}$, then
\begin{equation*}
	\big\Vert e^{\varepsilon|\cdot|_\dist/\sqrt{t}}K_{e^{-t\Delta}}\big\Vert_{L^1(G)}\lesssim_\varepsilon 1,\qquad \big\Vert e^{\varepsilon|\cdot|_\dist/\sqrt{t}} X_j K_{e^{-t\Delta}}\big\Vert_{L^1(G)}\lesssim_\varepsilon t^{-1/2}, \qquad t>0.
\end{equation*}
Moreover, for all $t_0>0$ and $\alpha,\beta \in \bigcup_{N \in \NN} \{0,\dots,d\}^N$,
\begin{equation}\label{eq:9}
	\big\Vert e^{\varepsilon|\cdot|_\dist/\sqrt{t}} X^\alpha (X^\beta K_{e^{-t\Delta}})^\ast\big\Vert_{L^1(G)}\lesssim_{\varepsilon,t_0,\alpha,\beta} t^{-(|\alpha|+|\beta|)/2},\qquad t\in(0,t_0].
\end{equation} 
\end{prop}

\section{Riesz transforms}\label{S:Riesz_tr}

\subsection{Decomposition and boundedness of the local parts}

We begin by recalling from \cite{MaVa} the precise definition of the Riesz transforms $\Riesz_j$. By \cite[Proposition~4.1]{MaVa} we know that the range of $\Delta^{-1/2}$ is contained in the domain of each $X_j$, $j=0,\ldots,d$. Moreover, $X_j\Delta^{-1/2}$ defined initially on the domain of $\Delta^{-1/2}$ extends to a bounded operator $\overline{X_j\Delta^{-1/2}}$ on $L^2$. We define $\Riesz_j = \overline{X_j\Delta^{-1/2}}$.

We recall the well-known subordination formula
\begin{equation*}
	\Delta^{-1/2} =\frac{2}{\pi}\int_0^\infty \big(t^2+\Delta\big)^{-1}\dd t,
\end{equation*}
which allows us to write, at least on a dense domain,
\begin{equation*}
	\Riesz_j =\frac{2}{\pi}\int_0^\infty X_j \big(t^2+\Delta\big)^{-1}\dd t.
\end{equation*}
Thus, we can decompose the Riesz transforms as
\begin{equation}\label{eq:Riesz_dec}
\frac{\pi}{2} \Riesz_j = \Riesz_j^{0}+\Riesz_j^{\infty},
\end{equation}
where
\begin{equation}\label{eq:Riesz_loc_inf}
	\Riesz_j^{0} =   \int_{1/2}^\infty X_j (t^2+\Delta)^{-1} \,\dd t, \qquad
	\Riesz_j^{\infty} =   \int_0^{1/2} X_j (t^2+\Delta)^{-1} \,\dd t.
\end{equation}
The reason for the splitting of the integral at the point $1/2$ is merely technical, in that it simplifies the presentation of some of the estimates later in the proof.

Each of the above operators is bounded on $L^2(G)$. We shall justify that they are also bounded on $L^p(G)$ for $1<p<\infty$.
As we shall see, the ``local parts'' $\Riesz_j^{0}$ can be treated in a relatively standard way by means of Theorem \ref{thm:CZ}: indeed, for these parts, favourable estimates for the corresponding convolution kernels can be derived from the heat kernel estimates in Proposition \ref{prop:MaVa}.
Justifying the $L^p$-boundedness of the ``parts at infinity'' $\Riesz_j^\infty$ will require a different approach, which is the main contribution of this paper.

Let us start with the discussion of the local parts.

\begin{prop}\label{prop:localparts}
	The operators $\Riesz_{j}^{0}$, $j=0,\ldots,d$, are bounded on $L^p(G)$ for $1 < p < \infty$.
\end{prop}
\begin{proof}
We shall make use of the subordination formula
\begin{equation*}
	(t^2+\Delta)^{-1} =\int_0^\infty e^{-s(t^2+\Delta)} \,\dd s
\end{equation*}
to reduce the problem to heat kernel estimates. It shall actually be convenient to further split $\Riesz_j^0 = \Riesz_j^{0,0} + \Riesz_j^{0,\infty}$, where
\[
\Riesz_j^{0,0} =\int_{1/2}^\infty\int_0^1 e^{-st^2} X_j e^{-s\Delta}\,\dd s\,\dd t, \qquad
\Riesz_j^{0,\infty} =\int_{1/2}^\infty\int_1^\infty e^{-st^2} X_j e^{-s\Delta}\,\dd s\,\dd t.
\]
Clearly, analogous integral formulas hold true for the respective convolution kernels.

The operators $\Riesz_j^{0,\infty}$ are bounded on $L^p(G)$ for all $p \in [1,\infty]$, because their convolution kernels are in $L^1(G)$. Indeed,
by Proposition \ref{prop:MaVa},
	\begin{multline*}
		\big\Vert K_{\Riesz_j^{0,\infty}} \big\Vert_{L^1(G)}\leq   \int_{1/2}^\infty\int_1^\infty e^{-st^2} \big\Vert X_j K_{e^{-s\Delta}} \big\Vert_{L^1(G)}\,\dd s\,\dd t\\
		 \lesssim \int_{1/2}^\infty\int_1^\infty e^{-st^2} \,\frac{\dd s}{\sqrt{s}}\,\dd t<\infty. 
	\end{multline*}

It remains to justify the required boundedness properties of $\Riesz_{j}^{0,0}$; for this we shall instead apply Theorem \ref{thm:CZ}.
More precisely, let us fix $j\in \{ 0,\ldots,d\}$. We shall show that both $\Riesz_j^{0,0}$ and its adjoint $(\Riesz_j^{0,0})^\ast$ satisfy the assumptions of Theorem \ref{thm:CZ} with $c=\sqrt{2}$ and $\varepsilon=1$, hence $\Riesz_j^{0,0}$ is bounded on $L^p(G)$ for all $p\in (1,\infty)$.

Much as in \cite[proof of Theorem 1.1(i)]{MaVa}, we decompose
	\begin{equation*}
	\Riesz_{j}^{0,0} =\sum_{n\in\NN} \int_{2^{-n-1}}^{2^{-n}} X_j e^{-s\Delta} \int_{1/2}^\infty e^{-st^2}\,\dd t\,\dd s  =: \sum_{n\in\NN} T_{n}.
	\end{equation*}
	Clearly, $T_n$ is a convolution operator associated with the kernel
	\begin{equation*}
	K_n = \int_{2^{-n-1}}^{2^{-n}} X_j K_{e^{-s\Delta}}\int_{1/2}^\infty e^{-st^2}\,\dd t\,\dd s.
	\end{equation*}
	Moreover,
	\begin{equation*}
	\big(\Riesz_{j}^{0,0}\big)^\ast = \sum_{n\in\NN} T^\ast_{n}, \qquad 	K_n^\ast= \int_{2^{-n-1}}^{2^{-n}} \big(X_j K_{e^{-s\Delta}}\big)^\ast\int_{1/2}^\infty e^{-st^2}\,\dd t\,\dd s.
	\end{equation*}
	
Observe that
	\begin{equation*}
	\int_{1/2}^\infty e^{-st^2}\,\dd t \simeq s^{-1/2}, \qquad s\in(0,1).
	\end{equation*}
	Thus, by applying Proposition \ref{prop:MaVa} we obtain
	\begin{equation*}
		\begin{split}
	\int_G \big| K_n(\bdx)\big| (1+ 2^{n/2}|\bdx|_\dist)\,\dd\bdx &= \int_G \big| K_n^\ast(\bdx)\big| (1+ 2^{n/2}|\bdx|_\dist)\,\dd\bdx\\
	&\lesssim \int_{2^{-n-1}}^{2^{-n}} \big\Vert e^{|\cdot|_\dist/\sqrt{s}} X_j K_{e^{-s\Delta}}\big\Vert_{L^1(G)} \frac{\dd s}{\sqrt{s}}\\
	&\lesssim \int_{2^{-n-1}}^{2^{-n}} \frac{\dd s}{s}\simeq 1.
	\end{split}
	\end{equation*}
On the other hand, for $k\in\{0,\ldots,d \}$, Proposition \ref{prop:MaVa} implies that
	\begin{equation*}
	\int_G \big| X_{k} K_n^\ast(\bdx)\big|\,\dd \bdx \leq \int_{2^{-n-1}}^{2^{-n}} \big\Vert X_{k}\big( X_j K_{e^{-s\Delta}}\big)^\ast\big\Vert_{L^1(G)} \frac{\dd s}{\sqrt{s}} \lesssim \int_{2^{-n-1}}^{2^{-n}} s^{-3/2} \,\dd s\simeq 2^{n/2},
	\end{equation*}
and much the same
	\begin{equation*}
	\int_G \big| X_{k} K_n(\bdx)\big|\,\dd \bdx \leq \int_{2^{-n-1}}^{2^{-n}} \big\Vert X_{k} X_j K_{e^{-s\Delta}}\big\Vert_{L^1(G)} \frac{\dd s}{\sqrt{s}} \lesssim \int_{2^{-n-1}}^{2^{-n}} s^{-3/2} \,\dd s\simeq 2^{n/2},
	\end{equation*}
as required.
\end{proof}

\begin{rem}
Notice that the argument from the proof of Proposition \ref{prop:localparts}, where the Riesz transforms are subordinated to the heat propagator, could also be applied to the parts at infinity $\Riesz_j^\infty$ if the heat kernel estimate \eqref{eq:9} held also for large times. As it turns out, \eqref{eq:9} actually holds uniformly in $t \in (0,\infty)$ for any $|\alpha|,|\beta| \leq 1$, whence the $L^p$-boundedness of the Riesz transforms $\Riesz_j$ for $1 < p < 2$ follows \cite{MaVa}. The case $2 < p < \infty$ would instead require a large-time version of \eqref{eq:9} for $|\alpha| = 2$ and $|\beta| = 0$; however, such estimate cannot be true, since it would also imply the boundedness of $\Riesz_j^\ast$ from $H^1(G)$ to $L^1(G)$, which is known to be false, at least when $N$ is abelian; see \cite[Proposition~7.1]{Ma23} and \cite[Theorems~4.2 and 5.2]{SjVa}. 
\end{rem}

\subsection{An operator-valued spectral multiplier theorem}

In order to prove that the parts at infinity $\Riesz_j^\infty$ are bounded on $L^p(G)$, we shall make use of the operator-valued spectral multiplier theorem proved in \cite{MaPl2}. 

For the reader's convenience, we briefly recall some facts discussed in detail in \cite[Section 7]{MaPl2}. Let $X$ be a $\sigma$-finite measure space and $L$ be a positive self-adjoint operator on $L^2(X)$. By the spectral theorem, there exist a $\sigma$-finite measure space $\Omega$, a measurable function $\ell\colon\Omega\to\Rpos$ and a unitary transformation $\Upsilon: L^2(X)\to L^2(\Omega)$ that intertwines $L$ with the operator of multiplication by $\ell$:
\begin{equation*}
\Upsilon (L f) (\omega) = \ell(\omega) \Upsilon f(\omega).
\end{equation*}
Moreover, for any bounded Borel function $F\colon \Rpos\to\CC$ there holds
\begin{equation*}
\Upsilon (F(L) f) (\omega) = F\big(\ell(\omega)\big) \Upsilon f(\omega),
\end{equation*}
and $F(L)$ is bounded on $L^2(X)$.

This construction can be extended to operator-valued functions. Namely, observe that $\tilde{\Upsilon}=\Upsilon\otimes \id : L^2(X \times \RR) \to L^2(\Omega \times \RR)$ is also a unitary transformation. Let $M: \Rpos\to\LinBnd(L^2(\RR))$ be a uniformly bounded, weakly measurable operator-valued function; then we define the operator $M(L) \in \LinBnd(L^2(X \times \RR))$ by
\begin{equation}\label{eq:opval_fc}
\tilde{\Upsilon}\big(M(L)f\big) (\omega,u)= \big(M(\ell(\omega)) \tilde{\Upsilon}f(\omega,\cdot)\big)(u)
\end{equation}
for all $f\in L^2(X\times\RR)$ and almost all $(\omega,u) \in \Omega \times \RR$.
One can show that the definition of $M(L)$ is independent of the intertwining transformation $\Upsilon$, and moreover
\begin{equation*}
\Vert M(L)\Vert_{\LinBnd(L^2(X\times\RR))} \leq \sup_{\xi\in\Rpos} \Vert M(\xi)\Vert_{\LinBnd(L^2(Y))}.
\end{equation*}

We recall some further terminology from \cite{MaPl2} that we shall frequently use in the sequel. An operator-valued function $M\colon\Rpos\to\LinBnd(L^2(\RR))$ is said to be \emph{(weakly) continuous} if for all $f,g\in L^2(\RR)$ the matrix coefficients
\begin{equation}\label{eq:matcoef}
\xi\mapsto \langle M(\xi)f,g\rangle
\end{equation} 
are continuous on $\Rpos$. Furthermore, for any $k \in \NN$, the function $M\colon\Rpos\to\LinBnd(L^2(\RR))$ is said to be \emph{(weakly) of class $C^k$} if for all $f,g\in L^2(\RR)$ the matrix coefficients \eqref{eq:matcoef} are in $C^k(\Rpos)$; under this assumption, for any $i=0,\dots,k$, there exists a continuous operator-valued function $\partial_\xi^i M : \Rpos \to \LinBnd(L^2(\RR))$ such that
\begin{equation*}
\langle \partial_\xi^i M(\xi)f,g\rangle = \partial_\xi^i \langle M(\xi)f,g\rangle
\end{equation*} 
for all $f,g \in L^2(\RR)$.
As usual, if $M$ is of class $C^k$ for all $k\in\NN$, then we say that $M$ is of class $C^\infty$.

\begin{rem}\label{rem:1}
An application of the Banach--Steinhaus theorem readily shows that, for any given $k \in \NN$, a function $M\colon\Rpos\to\LinBnd(L^2(\RR))$ is of class $C^k$ if and only if there exist continuous operator-valued functions $M_i : \Rpos \to \LinBnd(L^2(\RR))$, $i=0,1,\ldots,k$, and a dense subclass $\mathcal{A}\subset L^2(\RR)$ such that for all $f,g\in\mathcal{A}$ the matrix coefficients \eqref{eq:matcoef} are in $C^k(\Rpos)$ and there holds
	\begin{equation*}
	\partial_\xi^i \langle M(\xi)f,g\rangle = \langle M_i(\xi)f,g\rangle\qquad \forall f,g\in\mathcal{A}.
	\end{equation*}
\end{rem}

Let $\chi\in C_c^\infty(\Rpos)$ be a nontrivial nonnegatite cutoff. For $s>1/2$ we say that a Borel function $F\colon\Rnon\to\CC$ is in the \emph{H\"ormander class} $\hormander{s}$ if 
\begin{equation*}
\Vert F\Vert_{\hormander{s}} \defeq \sup_{t>0} \Vert F(t\cdot) \chi \Vert_{\sobolev{s}{2}}<\infty.
\end{equation*}
The definition of $\hormander{s}$ does not depend on the choice of $\chi$ and the resulting norms are comparable for all such $\chi$. Keep in mind that, by the Sobolev embedding theorem, elements of $\hormander{\alpha}$ are continuous and bounded on $\Rpos$.
We shall say that a positive self-adjoint operator $L$ on $L^2(X)$ has a \emph{bounded $\hormander{s}$ functional calculus on $L^p(X)$} for some $p \in (1,\infty)$ if, for any $F \in \hormander{s}$, the operator $F(L)$ is $L^p(X)$-bounded and
\[
\|F(L)\|_{p \to p} \lesssim \|F\|_{\hormander{s}}.
\]

Recall that a nonnegative locally integrable function $w$ is said to be in the \emph{Muckenhoupt class} $A_p(\RR)$, $p\in(1,\infty)$, if
\begin{equation*}
[w]_{A_p}\defeq \sup_{I} \left( \frac{1}{|I|} \int_I w\right) \left( \frac{1}{|I|} \int_I w^{-\frac{1}{p-1}}\right)^{p-1}<\infty,
\end{equation*}
where the supremum is taken over all intervals $I$ in $\RR$. The quantity $[w]_{A_p}$ is called  the $A_p$ characteristic of the weight $w$. We point out that the $A_p$ characteristic is translation-invariant, i.e., $[w(\cdot+t)]_{A_p} = [w]_{A_p}$ for all $t \in \RR$.

The following result is a particular case of \cite[Corollary 7.8]{MaPl2}.

\begin{thm}\label{thm:V-VSMT}
Let $s_L \geq 1/2$ and $p \in (1,\infty)$.
Let $L$ be a positive self-adjoint operator on $L^2(X)$ having a bounded $\hormander{s}$-calculus on $L^p(X)$ for all $s>s_L$. Let $N > s_L + 3/2$ be an integer. Let $M\colon\Rpos\to\LinBnd(L^2(\RR))$ be of class $C^{N}$. Assume that there exists a nondecreasing function $\psi : [1,\infty) \to [0,\infty)$ such that, for all $w\in A_2(\RR)$,
\[
\max_{j=0,\dots,N} \sup_{\xi\in\Rpos}\big\Vert \xi^j\partial_\xi^j M(\xi)\big\Vert_{\LinBnd(L^2(w))} \leq \psi([w]_{A_2}).
\]
Then, $M(L)$ is a bounded operator on $L^p(X\times \RR)$.
\end{thm}

\subsection{The parts at infinity of the Riesz transforms}

In this section we study the operators $\Riesz_j^\infty$. We begin with a reduction. Set
\begin{equation}\label{eq:tRiesz}
\begin{aligned}
\widetilde{\Riesz}_0^\infty &\defeq \Riesz_0^\infty - \big(\Riesz_0^\infty\big)^\ast=\int_0^{1/2}\big( X_0 (t^2+\Delta)^{-1}+(t^2+\Delta)^{-1}X_0\big) \,\dd t,\\
\widetilde{\Riesz}_1^\infty &\defeq  \int_0^{1/2} \big(e^u \Lapl^{1/2}\big) (t^2+\Delta)^{-1}\,\dd t.  
\end{aligned}
\end{equation}

In the next proposition we reduce the question of the $L^p$-boundedness of the operators $\Riesz^\infty_j$, $j=0,\ldots,d$, to that of the operators $\widetilde{\Riesz}^\infty_j$, $j=0,1$. 

\begin{prop}\label{prop:8}
If the $\widetilde{\Riesz}^\infty_j$, $j=0,1$, are bounded on $L^p(G)$ for all $p \in (1,\infty)$, then also the $\Riesz^\infty_j$, $j=0,\ldots, d$, are bounded on $L^p(G)$ for all $p \in (1,\infty)$.
\end{prop}
\begin{proof}
For $j=0$ we begin by recalling that $\Riesz_0$ is bounded on $L^p(G)$, $p\in(1,2]$ (see \cite{MaVa}). Thus, by difference, from \eqref{eq:Riesz_dec} and Proposition \ref{prop:localparts} we deduce that $\Riesz_0^\infty$ is also bounded on $L^p(G)$, $p\in(1,2]$. By duality, $(\Riesz_0^\infty)^\ast$ is bounded on $L^p(G)$, $p\in [2,\infty)$. Thus, it suffices to consider $\widetilde{\Riesz}^\infty_0$ to prove the boundedness of $\Riesz^\infty_0$ in the full range $p\in(1,\infty)$. 

For $j\in\{1,\ldots,d\}$ observe that, by \eqref{eq:Riesz_loc_inf} and \eqref{eq:tRiesz},
\begin{equation*}
\Riesz_j^\infty =  
X_j \Lapl^{-1/2} \int_0^{1/2} \Lapl^{1/2} (t^2+\Delta)^{-1}\, \dd t = 
\breve{X}_j \Lapl^{-1/2} \widetilde\Riesz_1^\infty.
\end{equation*}
Recall that the $\breve{X}_j \Lapl^{-1/2}$ are the Riesz transforms associated with $\Lapl$ and they are bounded on $L^p(N)$, $p\in(1,\infty)$; see \cite[Theorem~4.10]{Fo75} or \cite{ChGe,tElRoSi,LoVa}. Hence, in order to justify the boundedness of each $\Riesz^\infty_j$ on $L^p(G)$, $p\in(1,\infty)$, it suffices to show this for $\widetilde{\Riesz}_1^\infty $.
\end{proof}

Our goal is to write $\widetilde{\Riesz}_j^\infty$, $j=0,1$, as an operator-valued function of $\Lapl$. Mind that $L^2(G)=L^2(N\times \RR)$. For that purpose we introduce the family
\begin{equation}\label{eq:Hxi}
H(\xi) = -\partial_u^2 +e^{2u}\xi^2, \quad\xi\in\Rpos,
\end{equation}
of Schr\"odinger operators with positive potentials on $L^2(\RR)$. Notice that each $H(\xi)$ is a positive self-adjoint operator on $L^2(\RR)$ with core $C^\infty_c(\RR)$, see, e.g., \cite[Chapter 2, Theorem 1.1]{BS}. By comparing \eqref{eq:subLaplacian} and \eqref{eq:Hxi} we see that, at least formally, $\Delta=H(\sqrt{\Lapl})$. 
The following lemma shall make this relation precise, by considering the functional calculi of $\Delta$ and the $H(\xi)$, and making use of the operator-valued functional calculus of \eqref{eq:opval_fc}. We shall use the notation
\[
T_x f(y) = f(y+x), \qquad x,y\in\RR.
\] 

\begin{lm}\label{lm:6}
Let $F\colon \Rpos\to\CC$ be a bounded Borel function.
	\begin{enumerate}[label=(\roman*)]
		\item\label{lm:6_i:F(H(xi))} $F\big(H(\xi)\big)=T_{\log\xi}F\big(H(1)\big)T_{-\log\xi}$, $\xi\in\Rpos$.
		\item\label{lm:6_i:F(H)cont} The function $F \circ H : \Rnon \to \LinBnd(L^2(\RR))$ deifned by $(F \circ H)(\xi) = F(H(\xi))$ for all $\xi \in \Rpos$, is continuous and uniformly bounded.
		\item\label{lm:6_i:F(Delta)} $(F\circ H)(\sqrt{\Lapl}) = F(\Delta)$. 
	\end{enumerate}
\end{lm}
\begin{proof}
The relation
\[
H(\xi)=T_{\log\xi}H(1) T_{-\log\xi}, \quad \xi\in\Rpos,
\]
follows directly from the definition of $H(\xi)$; as $T_{\log \xi}$ is a unitary transformation on $L^2(\RR)$, by the uniqueness of the functional calculus we deduce part \ref{lm:6_i:F(H(xi))}.
	
For part \ref{lm:6_i:F(H)cont} observe that, by part \ref{lm:6_i:F(H(xi))}, for all $f,g\in L^2(\RR)$,
\begin{equation*}
\big\langle F(H(\xi))f,g\big\rangle = \big\langle F(H(1)) T_{-\log\xi}f,T_{-\log\xi} g\big\rangle.
\end{equation*}
Since translations are continuous and uniformly bounded in $L^2(\RR)$, we deduce the continuity and boundedness of $F \circ H$.

In order to justify part \ref{lm:6_i:F(Delta)}, by a density argument, it suffices to consider only $F(\lambda)=e^{-t\lambda}$, $t>0$, and prove that
\begin{equation}\label{eq:24}
	E_t(\sqrt{\Lapl}) f =e^{-t\Delta} f
\end{equation}
for all $t > 0$ and $f \in L^2(G)$, where $E_t(\xi)=e^{-tH(\xi)}$. Another density argument shows that it is enough to verify \eqref{eq:24} for $f=f_1\otimes f_2$, where $f_1\in L^2(N)$ and $f_2\in L^2(\RR)$.
	
	We now use Gnewuch's formula \cite[Theorem~2.1]{G}, which applied for $\Delta=-\partial_u^2 +e^{2u}\Lapl$ gives that
	\begin{equation*}
	e^{-t\Delta} (f_1 \otimes f_2) (x,u) = \int_{\RR} \int_0^\infty \psi_t(\zeta) e^{-\frac{\cosh(u-v)}{\zeta}} e^{-\frac{\zeta e^{u+v}}{2}\Lapl} f_1(x) \,\dd \zeta\ f_2(v)\, \dd v,
	\end{equation*}
	where $\psi_t$ is defined in \eqref{eq:psi}. Moreover, the same formula for $H(\xi)=-\partial_u^2 + e^{2u}\xi^2$ implies that
	\begin{equation*}
	E_t(\xi) f_2(u) = \int_{\RR} \int_0^\infty \psi_t(z) e^{-\frac{\cosh(u-v)}{\zeta}} e^{-\frac{\zeta e^{u+v}\xi^2}{2}}  \,\dd \zeta\ f_2(v)\, \dd v.
	\end{equation*}
	
Thus, by the definition \eqref{eq:opval_fc} of the operator-valued functional calculus,
	\begin{equation*}
	\begin{split}
	&\tilde{U}\big( E_t(\sqrt{\Lapl})(f_1\otimes f_2)\big)(\omega,u) \\
	&= Uf_1(\omega) E_t(\ell(\omega))  f_2(u)\\
	 &= Uf_1(\omega) \int_{\RR} \int_0^\infty \psi_t(z)  e^{-\frac{\cosh(u-v)}{\zeta}} 
	 e^{-\frac{\zeta e^{u+v}\ell(\omega)^2}{2}}  \,\dd \zeta\ f_2(v)\, \dd v\\
	 &= \int_{\RR} \int_0^\infty \psi_t(z)  e^{-\frac{\cosh(u-v)}{\zeta} }
	  U\big(e^{-\frac{\zeta e^{u+v}}{2}\Lapl} f_1\big)(\omega) \,\dd \zeta\ f_2(v)\, \dd v\\
	 &=\tilde{U} \big(e^{-t\Delta}(f_1\otimes f_2) \big)(\omega,u),
	 \end{split}
	\end{equation*}
as required.	
\end{proof}

We define the operator-valued functions $M_0,M_1$ on $\Rpos$ given by
\begin{align}
\begin{split}\label{eq:M_j-def}
M_0(\xi) &= \int_0^{1/2}  \big(\partial_u(t^2+H(\xi))^{-1} +(t^2+H(\xi))^{-1}\partial_u\big) \,\dd t,\\
M_1(\xi) &=  \int_0^{1/2} e^{u} \xi \big(t^2+H(\xi)\big)^{-1}\,\dd t.
\end{split}
\end{align}

\begin{prop}\label{prop:4}
	The functions $M_j\colon\Rpos\to\LinBnd(L^2(\RR))$ are continuous and uniformly bounded for $j=0,1$. Moreover,
	\begin{equation}\label{eq:31}
	\widetilde{\Riesz}_0^\infty=M_0(\sqrt{\Lapl}),\qquad \widetilde{\Riesz}_1^\infty = M_1(\sqrt{\Lapl}).
	\end{equation}
\end{prop}
\begin{proof}
First, we justify that both $M_j(1)$, $j=0,1$, are bounded on $L^2(\RR)$.
We argue much as in the proof of \cite[Proposition~4.1]{MaVa} and denote by $\Dom(T)$ the domain of an operator $T$ on $L^2(\RR)$. 

Write $H=H(1)$ for brevity, and recall that  both $\partial_u H^{-1/2}$ and $V^{1/2} H^{-1/2}$, where $V(u)=e^{2u}$, are bounded on $L^2(\RR)$ (even on $L^p(\RR)$, $p\in(1,\infty)$; see \cite[Theorem~1.3]{Ma23}). In particular $\Dom(H^{1/2}) \subseteq \Dom(\partial_u) \cap \Dom(V^{1/2})$ and the formula 
	\begin{equation*}
	\big\Vert H^{1/2} f\Vert_{L^2(\RR)}^2 = \big\Vert \partial_u f\Vert_{L^2(\RR)}^2+\big\Vert V^{1/2} f\Vert_{L^2(\RR)}^2,
	\end{equation*}
	holds for all $f\in \Dom(H^{1/2})$.
	
Let $F$ be a nonnegative Borel function on $(0,\infty)$ such that $F(\lambda) \leq \lambda^{-1/2}$ for all $\lambda \in (0,\infty)$. Since, by the properties of the spectral calculus, for any $f\in L^2(\RR)$ we have $F(H)f\in\Dom(H^{1/2})$, we deduce that
	\begin{equation*}
	\big\Vert \partial_u (F(H) f)\Vert_{L^2(\RR)},\big\Vert V^{1/2} F(H) f\Vert_{L^2(\RR)}\leq \Vert f\Vert_{L^2(\RR)}, \qquad f\in L^2(\RR).
	\end{equation*}
	
Recall that
	\begin{equation*}
	\lambda^{-1/2} = \frac{2}{\pi}\int_0^\infty (t^2+\lambda)^{-1}\,\dd t.
	\end{equation*}
In particular, if we take
	\begin{equation*}
	F(\lambda) =\frac{2}{\pi} \int_0^{1/2} (t^2+\lambda)^{-1}\,\dd t,
	\end{equation*}
then $F(\lambda) \leq \lambda^{-1/2}$ as required. Thus, we see that the operators 
	\begin{equation*}
	\partial_u F(H) = \int_0^{1/2} \partial_u (t^2+H)^{-1}\,\dd t,\qquad V^{1/2} F(H) = \int_0^{1/2} e^{u} (t^2+H)^{-1}\,\dd t
	\end{equation*}
	are bounded on $L^2(\RR)$. Since
	\begin{equation*}
	\left(\int_0^{1/2} \partial_u (t^2+H)^{-1}\,\dd t\right)^\ast = - \int_0^{1/2} (t^2+H)^{-1}\partial_u\,\dd t
	\end{equation*}
	we conclude that the $M_j(1)$ are bounded on $L^2(\RR)$.
	
	By Lemma \ref{lm:6}\ref{lm:6_i:F(H(xi))} and the fact that 
	\begin{equation*}
	S(\xi) f = T_{\log \xi} S(1) T_{-\log \xi},
	\end{equation*}
	where $S(\xi)$ is either of the two operators $f(u)\mapsto f'(u)$ or $f(u)\mapsto \xi e^u f(u)$, we obtain that
	\begin{equation}\label{eq:26}
	M_j(\xi) = T_{\log \xi} M_j(1) T_{-\log \xi},\qquad \xi> 0,\ j=0,1.
	\end{equation}
	Thus, arguing as in the proof of Lemma \ref{lm:6}\ref{lm:6_i:F(H)cont}, we see that the $M_j(\xi)$ are continuous and uniformly bounded. Consequently, the operators $M_j(\sqrt{\Lapl})$, $j=0,1$, defined as in \eqref{eq:opval_fc}, are bounded on $L^2(G)$.
	
	Finally, notice that, by Lemma \ref{lm:6}\ref{lm:6_i:F(Delta)},
	\begin{equation*}
	\int_{\varepsilon}^{1/2} \big(t^2+H(\sqrt{\Lapl})\big)^{-1} \,\dd t = \int_{\varepsilon}^{1/2} ( t^2 + \Delta)^{-1} \,\dd t
	\end{equation*}
  for all $\varepsilon \in (0,1/2)$.
	Thus, at least on dense domain, we also have
\begin{align*}
\int_{\varepsilon}^{1/2} e^u\sqrt{\Lapl}\big(t^2+H(\sqrt{\Lapl})\big)^{-1} \,\dd t &= \int_{\varepsilon}^{1/2} e^u \sqrt{\Lapl}( t^2 + \Delta)^{-1} \,\dd t,\\
\int_{\varepsilon}^{1/2} \partial_u\big(t^2+H(\sqrt{\Lapl})\big)^{-1} \,\dd t &= \int_{\varepsilon}^{1/2} \partial_u ( t^2 + \Delta)^{-1} \,\dd t,\\
\int_{\varepsilon}^{1/2} \big(t^2+H(\sqrt{\Lapl})\big)^{-1}\partial_u \,\dd t &= \int_{\varepsilon}^{1/2} ( t^2 + \Delta)^{-1}\partial_u \,\dd t.
\end{align*}
If we take the limit
as $\varepsilon \to 0^+$, by comparing \eqref{eq:tRiesz} and \eqref{eq:M_j-def} we obtain \eqref{eq:31}.
\end{proof}

The identities \eqref{eq:31} show that the operators $\widetilde{\Riesz}^\infty_j$, to which we reduced (see Proposition \ref{prop:8}) the study of the Riesz transforms, are particular operator-valued functions of $\sqrt{\Lapl}$, corresponding to the operator-valued multipliers $M_j$. As $\sqrt{\Lapl}$ has an $\hormander{s}$-bounded functional calculus on $L^p(N)$ for any $p \in (1,\infty)$ and $s > Q/2$ \cite{Ch,MaMe}, we can apply Theorem \ref{thm:V-VSMT} to reduce the $L^p(G)$-boundedness 
of the $\widetilde{\Riesz}^\infty_j$ to appropriate bounds on the $M_j$ and their homogeneous derivatives. Namely,  in view of Theorem \ref{thm:V-VSMT} it suffices to justify that 
\begin{equation}\label{eq:23}
\sup_{\xi\in\Rpos} \big\Vert (\xi\partial_\xi)^n M_j(\xi)\big\Vert_{\LinBnd(L^2(w))} \lesssim_{n,[w]_{A_2}} 1
\end{equation}
for $j=0,1$, all $n \in \NN$ and all $w \in A_2(\RR)$.

\section{Kernel estimates for the operator-valued multipliers}\label{S:M_j-estimates}

In this section we shall prove \eqref{eq:23} and, consequently, Theorem \ref{thm:main}. In order to do so we will justify that the $M_j(\xi)$ are integral operators associated with kernels satisfying appropriate bounds.

Let $M_j(\xi,t)$, for $t\in(0,1/2)$, $\xi\in\Rpos$, and $j=0,1$, be the $L^2(\RR)$-bounded operators given by
\begin{equation}\label{eq:def_Mjt}
\begin{aligned}
M_0(\xi,t) &=\partial_u \big(t^2+H(\xi)\big)^{-1}+\big(t^2+H(\xi)\big)^{-1}\partial_u,\\
M_1(\xi,t) &= e^u \big(t^2+H(\xi)\big)^{-1},
\end{aligned}
\end{equation}
so that we have
\begin{equation}\label{eq:Mj_int}
M_j(\xi) = \int_0^{1/2} M_j(\xi,t)\,\dd t.
\end{equation}

It is known (see \cite{Hu} or \cite[\S 4.15]{Ti}) that the resolvent operators $(t^2+H(1))^{-1}$ are integral operators, whose kernels are explicitly given by
\begin{equation}\label{eq:Rsl}
\iker_{(t^2+H(1))^{-1}}(u,v) = \Rsl_t(u,v)\defeq I_t\big(e^{\min(u,v)}\big) K_t\big( e^{\max(u,v)}\big),
\end{equation}
where $I_t$ and $K_t$ are modified Bessel functions (see Section \ref{S:Appendix} below for details). These formulas are the starting point of our analysis.

\begin{prop}
The operators $M_j(\xi,t)$, $\xi\in\Rpos$, $t\in(0,1)$, are integral operators, whose integral kernels satisfy
\begin{equation}\label{eq:32}
\iker_{M_j(\xi,t)}(u,v) = \iker_{M_j(1,t)}(u+\log\xi,v+\log\xi).
\end{equation}
Moreover,	
\begin{equation}\label{eq:33}
\iker_{M_1(1,t)}(u,v)= e^u \Rsl_t(u,v) = e^u I_t\big(e^{\min(u,v)}\big) K_t\big( e^{\max(u,v)}\big)
\end{equation}
and
\begin{equation}\label{eq:30}
\begin{split}
	\iker_{M_0(1,t)}(u,v) &= (\partial_u-\partial_v) \Rsl_t(u,v) \\
	&= \sign(v-u) \Big(e^{\min(u,v)} I_{t+1}\big(e^{\min(u,v)}\big) K_{t}\big( e^{\max(u,v)}\big) \\
	&\qquad\qquad\qquad+ e^{\max(u,v)} I_t\big(e^{\min(u,v)}\big) K_{t+1}\big( e^{\max(u,v)}\big)\Big).
\end{split}
\end{equation}
Furthermore, the kernels $\iker_{M_j(\xi,t)}(u,v)$ are smooth in the direction $\partial_u+\partial_v$.
\end{prop}
\begin{proof}
	Observe that \eqref{eq:33} follows immediately from \eqref{eq:def_Mjt} and \eqref{eq:Rsl}. In a similar way, since $\min(u,v)$ and $\max(u,v)$ are Lipschitz functions of $(u,v)$, \eqref{eq:30} follows from the identity \eqref{eq:1st_der_IK} via the change of variables $x = e^{\min(u,v)}$, $y=e^{\max(u,v)}$.
	
Notice that much as in \eqref{eq:26} we have
\[
	M_j(\xi,t) = T_{\log\xi} M_j(1,t)T_{-\log\xi},\qquad \xi\in\Rpos,\ j=0,1,\ t\in(0,1),
\]
whence \eqref{eq:32} follows.
	
Finally, since the direction of the derivative $\partial_u+\partial_v$ is parallel to the diagonal, immediately from \eqref{eq:33}, \eqref{eq:30}, and in view of \eqref{eq:Bessel_rec}, we obtain the smoothness of $\iker_{M_j(1,t)}(u,v)$ in the direction $\partial_u+\partial_v$. By \eqref{eq:32} the same is true for $\iker_{M_j(\xi,t)}(u,v)$.
\end{proof}

Define
\begin{equation}\label{eq:27}
	S_j^n (t,u,v) \defeq (\partial_u+\partial_v)^n \iker_{M_j(1,t)}(u,v).
\end{equation}
Notice that, by \eqref{eq:32},
\begin{equation}\label{eq:Sjnt_der}
S_j^n(t,u+\log\xi,v+\log\xi) = (\xi\partial_\xi)^n \iker_{M_j(\xi,t)}(u,v).
\end{equation}
In the following lemma we give pointwise bounds for the kernels $S_j^n (t,u,v)$.

\begin{lm}\label{lm:3}
Let $c>0$. If $(n,j)\in\NN\times\{0,1\}\setminus\{(0,0) \}$, then 
\begin{equation}\label{eq:est_S_nj}
\begin{split}
	\big|S_j^n(t,u,v) \big|\lesssim_{n,c} 
  \begin{cases}
	1 &\text{if } |u-v|\leq c,\\
	e^{-(u+v)/2}  &\text{if }  |u-v|>c,\ u,v>0,\\
	e^{-e^{\max(u,v)}/2} e^{-t\left|\min(u,v)\right|}  &\text{if }  |u-v|>c,\ uv\leq 0,  \\
	e^{t(u+v)}+ \big(e^{-|u|} + e^{-|v|}\big) e^{-t|u-v|} &\text{if }  |u-v|>c,\ u,v<0,
  \end{cases}
\end{split}
\end{equation}
	uniformly in $t\in(0,1/2)$. Moreover, for $(n,j)=(0,0)$ we have
\begin{equation}\label{eq:est_S_00}
\begin{split}
	\big| S_0^0(t,u,v)\big|\lesssim_{c} 
	\begin{cases}
	1 &\text{if }  |u-v|\leq c,\\
	e^{-(u+v)/2} &\text{if }  |u-v|>c,\ u,v>0,\\
	e^{-e^{\max(u,v)}/2} e^{-t\left|\min(u,v)\right|} &\text{if }  |u-v|>c,\ uv\leq 0,  \\
	 e^{-t|u-v|} &\text{if }  |u-v|>c,\ u,v<0,
	\end{cases}
\end{split}
\end{equation}
	uniformly in $t\in(0,1/2)$.
\end{lm}
\begin{proof}
Let us first discuss a preliminary reduction for the estimates \eqref{eq:est_S_nj}. For $n\in \NN$, by \eqref{eq:27} and \eqref{eq:33} we have
\begin{equation*}
\big|S_1^n(t,u,v)\big|\lesssim_n \sum_{k=0}^n  e^u\Big|\big(\partial_u+\partial_v\big)^k \Rsl_t(u,v)\Big|.
\end{equation*}
Moreover, from \eqref{eq:Rsl} and \eqref{eq:2nd_der_IK} we deduce that
\begin{equation*}
	(\partial_u +\partial_v)(\partial_u - \partial_v) \Rsl_t(u,v)= \big(e^{2u}-e^{2v}\big) \Rsl_t(u,v);
\end{equation*}
thus, if $n\geq 1$, then by \eqref{eq:27} and \eqref{eq:30} there holds
	\begin{multline*}
	\big|S_0^n(t,u,v)\big| = \Big|\big(\partial_u+\partial_v\big)^{n-1} (e^{2u}-e^{2v}) \Rsl_t(u,v)\Big|\\
	 \lesssim_n  \sum_{k=0}^{n-1} \big|e^{2u}-e^{2v}\big| \Big|\big(\partial_u+\partial_v\big)^k \Rsl_t(u,v)\Big|.
	\end{multline*}
Consequently, if we set	
\begin{equation*}
	I_n(t,u,v)\defeq \big|(\partial_u+\partial_v)^n \Rsl_t(u,v)\big|, \quad 	
	I_n'(t,u,v)\defeq (1+|e^u-e^v|) (e^u+e^v) I_n(t,u,v),
\end{equation*}
then the estimates \eqref{eq:est_S_nj} for the kernels $S^n_j(t,u,v)$ are reduced to analogous estimates for the quantities $I_n'(t,u,v)$.

We shall now prove the estimates \eqref{eq:est_S_nj} for the kernels $I_n'(t,u,v)$, together with the estimates \eqref{eq:est_S_00} for the kernel $S_0^0(t,u,v) = \iker_{M_0(1,t)}(u,v)$ given by \eqref{eq:30}. By symmetry, it is enough to consider the case $v > u$. Recall throughout that we assume $t \in (0,1/2)$ and we require uniform estimates in $t$.
	
To deal with the first two cases in the estimates, we shall apply Lemma \ref{lm:1}. Notice that $(\partial_u+\partial_v)^n\Rsl_t$ corresponds to $(x\partial_x+y\partial_y)^n I_t(x)K_t(y)$ with the substitution $x=e^u$, $y=e^v$. Therefore, \eqref{eq:5} gives
\begin{equation*}
	I_n(t,u,v) \lesssim_n (1+e^v-e^u)^{n+1} e^{e^u-e^v} e^{-(u+v)/2}
\end{equation*}
and
\begin{equation}\label{eq:8}
\begin{split}
	I_n'(t,u,v)
	&\lesssim_n (e^u+e^v)  (1+e^v-e^u)^{n+2} e^{e^u-e^v} e^{-(u+v)/2}\\
	&\lesssim_n \cosh\left( \frac{v-u}{2} \right) e^{-(e^v-e^u)/2}
\end{split}
\end{equation}
for $u < v$ and $t\in(0,1/2)$.
Clearly, in the region $v-u\leq c$, this gives 
\begin{equation*}
	I_n'(t,u,v)\lesssim_{n,c} 1.
\end{equation*}
Moreover, if $v-u\geq c$ and $u,v>0$, then, by using the fact that
\begin{equation*}
	e^v-e^u = 2e^{(u+v)/2} \sinh\left(\frac{v-u}{2}\right),
\end{equation*}
from \eqref{eq:8} we obtain that
\begin{equation*}
 I_{n}'(t,u,v) \lesssim_{n}  e^{-(u+v)/2}\coth\left( \frac{v-u}{2}\right) \lesssim_c e^{-(u+v)/2}.
\end{equation*}
This proves the first two cases of \eqref{eq:est_S_nj} for $I_n'(t,u,v)$.	
For the kernel $S_0^0(t,u,v)$, instead, we apply \eqref{eq:6} with $N=0$ to obtain, much as above, that
\begin{equation*}
\big| S_0^0(t,u,v)\big| \lesssim \cosh\left(\frac{v-u}{2}\right)  e^{-(e^v-e^u)/2},
\end{equation*}
whence we deduce the first two cases of \eqref{eq:est_S_00}.
	
Now we consider the third region, where $u,v \in \RR$ satisfy $u \leq 0 \leq v$ and $v-u>c$. In particular we have $e^v \pm e^u \simeq_c e^v \geq 1$. Thus, by the first bound in Lemma \ref{lm:2} with $\varepsilon=1/4$ we get
\begin{equation*}
\begin{split}
I_n'(t,u,v) \simeq_c e^{2v} I_n(t,u,v) 
&\lesssim_{n} e^{2v} e^{(u+v)t} e^{-3(e^v-e^u)/4} \\
&\lesssim e^{5v/2-3e^v/4} e^{ut} \lesssim  e^{-e^v/2} e^{tu},
\end{split}
\end{equation*}
where we used that $u \leq 0 \leq v$ and $t \in (0,1/2)$. This proves the third case in \eqref{eq:est_S_nj} for $I_n'(t,u,v)$.
On the other hand, by the second bound in Lemma \ref{lm:2}, applied with $N=0$ and $\varepsilon=1/4$, we also deduce that
\begin{equation*}
	\big|S^0_0(t,u,v)\big| \lesssim_c  e^v e^{(u+v)t} e^{-3(e^v-e^u)/4}  \lesssim e^{-e^v/2} e^{tu},
\end{equation*}
i.e., the third case in \eqref{eq:est_S_00}.
	
We are left with the last case, namely the region where $u<v<0$ and $v-u>c$. We remark that we cannot apply Lemma \ref{lm:2} here, as $e^v - e^u\simeq_c e^v$ may be arbitrarily small in this region. 
On the other hand, from \eqref{eq:higher_der_I_K} we deduce that
\begin{equation*}
	I_n(t,u,v) \lesssim_n \sum_{k,\ell=0}^n e^{k u+\ell v} I_{t+k}\big( e^u \big) K_{t+\ell}\big( e^v \big).
\end{equation*}
Notice now by Lemma \ref{lm:5} that, for all $u,v \leq 0$ and $t \in (0,1/2)$,
\begin{equation}\label{eq:13}
	e^{ku+\ell v} I_{t+k}\big( e^{u}\big) K_{t+\ell}\big( e^{v}\big)
	\lesssim_{k,\ell} \begin{cases} 
	e^{2ku}  e^{-(v-u)t} &\text{if } \ell \neq 0,\\
	e^{2k u-v}  e^{(u+v)t} &\text{if } \ell = 0.
	\end{cases}
\end{equation}
Thus, if $u < v < 0$ and $v-u>c$, then
\begin{equation*}
	I_n'(t,u,v) \simeq e^v I_n(t,u,v) \lesssim e^{t(u+v)}+ \big(e^{-|u|} + e^{-|v|}\big) e^{-t|u-v|},
\end{equation*}
which proves the last case of \eqref{eq:est_S_nj}. 

It remains to prove the last case of \eqref{eq:est_S_00}. Starting from \eqref{eq:30}, if we use \eqref{eq:13} for $(k,\ell) = (1,0)$ and $(k,\ell) = (0,1)$, we obtain, much as above,
\begin{equation*}
	\big| S_0^0(t,u,v)\big| \lesssim e^{2u-v} e^{t(u+v)}+  e^{-t|u-v|}\lesssim e^{-t|u-v|},
\end{equation*}
as required.
\end{proof}

Let $n\in\NN$ and $j\in\{0,1\}$. We define
\[
S_j^n(u,v) \defeq \int_0^{1/2} S_j^n(t,u,v)\,\dd t.
\]
From \eqref{eq:Mj_int} and \eqref{eq:Sjnt_der} we deduce that
\begin{equation}\label{eq:Snj_kerder}
S_j^n(u+\log\xi,v+\log\xi) = (\xi\partial_\xi)^n \iker_{M_j(\xi)}(u,v).
\end{equation}
As we shall explain in Proposition \ref{prop:5}, the kernels $(u,v) \mapsto S_j^n(u+\log\xi,v+\log\xi)$ are actually the integral kernels of the operators $(\xi\partial_\xi)^n M_j(\xi)$.

Before that, we transfer the bounds from Lemma \ref{lm:3} for the functions $S_j^n(t,u,v)$ to the functions $S_j^n(u,v)$.

\begin{cor}\label{cor:1}
Let $n\in\NN$, $j\in\{0,1\}$, and $c>0$. If $(n,j)\neq (0,0)$, then
\begin{equation}\label{eq:12}
\begin{split}
\big|S_j^n(u,v)\big|\lesssim_{n,c} 
\begin{cases}
	1, & |u-v|\leq c,\\
	e^{-(u+v)/2} ,& |u-v|>c,\ u,v>0,\\
	\frac{e^{-e^{\max(u,v)}/2}}{\left|\min(u,v)\right|+1} ,& |u-v|>c,\ uv\leq 0,  \\
	\frac{1}{|u|+|v|+1}+ \frac{e^{-|u|} + e^{-|v|}}{|u-v|+1} ,& |u-v|>c,\ u,v<0,
\end{cases}
\end{split}
\end{equation}
whereas if $(n,j)=(0,0)$, then
\begin{equation*}
\begin{split}
	\big|S_0^0(u,v)\big|\lesssim_{n,c} 
\begin{cases}
	1, & |u-v|\leq c,\\
	e^{-(u+v)/2} ,& |u-v|>c,\ u,v>0,\\
	\frac{e^{-e^{\max(u,v)}/2}}{\left|\min(u,v)\right|+1} ,& |u-v|>c,\ uv\leq 0,  \\
	 \frac{1}{|u-v|+1} ,& |u-v|>c,\ u,v<0.
\end{cases}
\end{split}
\end{equation*}
\end{cor}

\begin{prop}\label{prop:6}
Let $n\in\NN$ and $j\in\{0,1\}$. The integral operator with integral kernel $(u,v) \mapsto S_j^n(u,v)$ is bounded on $L^2(\RR)$. Moreover, if $(n,j)\neq (0,0)$, the same operator is also bounded on $L^2(w)$ for any weight $w\in A_2(\RR)$, with a bound only depending on $n$ and $[w]_{A_2}$.
\end{prop}
\begin{proof}
Firstly we consider the case $(n,j)\neq (0,0)$.	Observe that by applying Corollary \ref{cor:1} with $c=1$ we obtain
\begin{equation*}
	\big| S^n_j(u,v)\big|\lesssim_n \ind_{\{|u-v|\leq 1\} } + \frac{1}{|u|+|v|+1} + \frac{e^{-|u|} + e^{-|v|}}{|u-v|+1}\ind_{\{|u-v|>1\}}.
\end{equation*}
Thus, it is enough to prove that the integral operators associated with each of these three kernels in the right-hand side of the above estimate are bounded on $L^2(w)$, with a bound only depending on $[w]_{A_2}$. The last two kernels are covered in \cite[Lemma~5.7]{Ma23}. For the operator associated with $\ind_{\{|u-v|\leq 1\} }$ notice that this is just the convolution operator with $\ind_{[-1,1]}$, which is trivially bounded by the Hardy--Littlewood maximal operator:
\begin{equation*}
	\big|f\ast \ind_{[-1,1]}\big| \leq 2 \mathcal{M}(|f|), 
\end{equation*}
Since $\mathcal{M}$ is bounded on $L^2(w)$ for all $w\in A_2(\RR)$, with a bound only depending on $[w]_{A_2}$ (see \cite[Theorem~7.1.9]{Gra1}), we conclude that the operator associated with $\ind_{\{|u-v|\leq 1\} }$ has this property as well.
	
For $(n,j) = (0,0)$, it simply remains to observe that, by \eqref{eq:Snj_kerder}, the kernel $S_0^0(u,v)$ is the integral kernel of $M_0(1)$, which we already know from Proposition \ref{prop:4} to be bounded on $L^2(\RR)$.
\end{proof}

\begin{prop}\label{prop:5}
For $j=0,1$, the operator-valued function $\xi\mapsto M_j(\xi)$ is of class $C^\infty$. Moreover, for any $n\in\NN$ and $\xi\in\Rpos$, the operators $(\xi\partial_\xi)^n M_j(\xi)$ are integral operators, whose integral kernels are given by
\[
	\iker_{(\xi\partial_\xi)^n M_j(\xi)} (u,v) = S_j^n (u+\log\xi ,v+\log\xi).
\]
In particular, if $(n,j)\neq(0,0)$, then $(\xi\partial_\xi)^n M_j(\xi)$ is bounded on $L^2(w)$ for any $w\in A_2(\RR)$, with a bound depending only on $n$ and $[w]_{A_2}$.
\end{prop}
\begin{proof}
For all $(n,j) \in \NN^2$ and $\xi \in \Rpos$, let $M_j^n(\xi)$ be the integral operator with integral kernel $(u,v) \mapsto S_j^n(u+\log\xi,v+\log\xi)$. Notice that, by construction,
\begin{equation}\label{eq:Mjn_transl}
M_j^n(\xi) = T_{\log \xi} M_j^n(1) T_{-\log \xi}.
\end{equation}
Thus, from Proposition \ref{prop:6}, arguing much as in the proof of Lemma \ref{lm:6}\ref{lm:6_i:F(H)cont}, we see that $M_j^n : \Rpos \to \LinBnd(L^2(\RR))$ is continuous and uniformly bounded.

Now, for all $f,g\in C_c(\RR)$, from \eqref{eq:Snj_kerder} we deduce that
\[
(\xi\partial_\xi)^n \langle M_j(\xi)f,g\rangle = \langle M_j^n(\xi)f,g\rangle,
\]
where the bounds from Corollary \ref{cor:1} are used to justify differentiating under the integral sign.
Thus, by Remark \ref{rem:1} we obtain that $M_j$ is of class $C^\infty$ and $(\xi\partial_\xi)^n M_j(\xi)=M_j^n(\xi)$ for all $n \in \NN$ and $j=0,1$.

Finally, for $(n,j)\neq (0,0)$, the desired weighted $L^2$-bounds for the operators $(\xi \partial_\xi)^n M_j(\xi) = M_j^n(\xi)$ follow from the corresponding bounds for the operators $M_j^n(1)$ proved in Proposition \ref{prop:6}, due to the relation \eqref{eq:Mjn_transl} and the fact that the $A_2$ characteristic is translation-invariant.
\end{proof}

We are left with proving the weighted $L^2$-bound for $M_0(\xi)$. Before we do that, we briefly recall some facts about classical Calder\'on--Zygmund operators on $\RR$ (see for instance \cite[Section~7.4]{Gra1}).

We say that a kernel $K : \RR^2 \to \CC$ is a \emph{standard kernel} if it is differentiable off the diagonal of $\RR^2$ and satisfies the estimates
\begin{equation}\label{eq:StanKer}
\big| K(u,v)\big| \lesssim |u-v|^{-1}\qquad \text{and}\qquad \big|\partial_u K(u,v)\big|+\big|\partial_v K(u,v)\big|\lesssim |u-v|^{-2}
\end{equation}
for $u\neq v$. 
An operator $T$ bounded on $L^2(\RR)$ shall be called a \emph{standard Calder\'on--Zygmund operator} if there exists a standard kernel $K$ such that, for any compactly supported $f\in L^2(\RR)$,
\[
Tf(x) = \int_\RR K(x,y) f(y)\, \dd y,\qquad x\notin \supp f.
\]
We recall a classical result concerning weighted bounds for Calder\'on--Zygmund operators; see, e.g., \cite[Theorem~7.11]{Duo}.

\begin{thm}\label{thm:CZ_oper}
Standard Calder\'on--Zygmund operators are bounded on $L^p(w)$ for all $w\in A_p(\RR)$, $p\in(1,\infty)$, with a bound only depending on $p$ and $[w]_{A_p}$. 
\end{thm}

In order to apply the above result to the operator $M_0(1)$, we need a further decomposition of its integral kernel.

\begin{prop}\label{prop:7}
There exist kernels $K_1, K_2$ such that 
\begin{equation}\label{eq:S00_dec}
	S_0^0(u,v)= K_1(u,v)+K_2(u,v),
\end{equation}
where $K_1$ is a standard kernel and $K_2$ satisfies an estimate analogous to \eqref{eq:12}, i.e., for any $c>0$,
\begin{equation}\label{eq:12_K2}
\begin{split}
	\big|K_2(u,v)\big|\lesssim_{c} 
	\begin{cases}
	1, & |u-v|\leq c,\\
	e^{-(u+v)/2} ,& |u-v|>c,\ u,v>0,\\
	\frac{e^{-e^{\max(u,v)}/2}}{\left|\min(u,v)\right|+1} ,& |u-v|>c,\ uv\leq 0,  \\
	\frac{1}{|u|+|v|+1}+ \frac{e^{-|u|} + e^{-|v|}}{|u-v|+1} ,& |u-v|>c,\ u,v<0.
	\end{cases}
\end{split}
\end{equation}
\end{prop}
\begin{proof}
We commence by observing that \eqref{eq:30} yields
\begin{equation}\label{eq:S0012}
\begin{split}
	S_0^0(u,v)
	 &= \sign(v-u) \int_0^{1/2} e^{\max(u,v)} I_{t}\big(e^{\min(u,v)}\big) K_{t+1}\big(e^{\max(u,v)}\big)\, \dd t\\
	  &\qquad+ \sign(v-u) \int_0^{1/2} e^{\min(u,v)}I_{t+1}\big(e^{\min(u,v)}\big) K_{t}\big(e^{\max(u,v)}\big)\,\dd t\\
	 &=: S_0^{0,1}(u,v)+S_0^{0,2}(u,v).
\end{split}
\end{equation}
	
Let $\chi\in C_c^\infty(\RR)$ be a smooth cutoff such that $0\leq \chi\leq 1$, $\chi$ is supported in $[1/4,4]$ and equal to $1$ in $[1/2,2]$; for convenience we shall also assume that $\chi(s) = \chi(1/s)$ for all $s > 0$. Then we obtain the decomposition \eqref{eq:S00_dec} if we set
\begin{align*}
	K_1(u,v) &:=S_0^{0,1}(u,v)\chi\Big(\frac{u}{v}\Big) \ind_{\RR^2_{-}}(u,v) ,\\
	K_2(u,v) &:=S_0^{0,1}(u,v)\Big(1-\chi\Big(\frac{u}{v}\Big)\Big) \ind_{\RR_{-}^2}(u,v) + S_0^{0,1}(u,v)\ind_{(\RR^2_{-})^c }(u,v)+S_0^{0,2}(u,v),
\end{align*}
where $\RR^2_-=\RR_{-}\times \RR_{-}$ and $\RR_{-} = (-\infty,0)$. We shall justify that $K_1$ and $K_2$ satisfy the required bounds. As before, by symmetry it is enough to consider the case $u < v$.
	
Firstly, notice that for $(u,v)\notin \RR^2_{-}$ we have $K_2(u,v)=S_0^0(u,v)$. Thus, Lemma \ref{lm:3} implies that $K_2$ satisfies \eqref{eq:12_K2} off $\RR^2_{-}$.

Moreover, if $u < v < 0$, then, by applying \eqref{eq:13} with $(k,\ell)=(1,0)$, 
\begin{equation*}
	\big| S_0^{0,2}(u,v)\big|\lesssim \frac{e^{2u-v}}{|u|+|v|+1} \lesssim \frac{1}{|u|+|v|+1},
\end{equation*}
while, by applying \eqref{eq:13} with $(k,\ell)=(0,1)$, 
\begin{equation}\label{eq:14}
	\big|S_0^{0,1}(u,v)\big| \lesssim \frac{1}{|u-v|+1}.
\end{equation}
We now notice that in the support of
\begin{equation*}
	(u,v)\mapsto (1-\chi(u/v))\ind_{\RR_{-}^2 }(u,v)
\end{equation*}
we have $|u-v|\simeq |u|+|v|$. Thus, the previous estimates show that $K_2$ satisfies \eqref{eq:12_K2} also in the region $\RR_-^2$.
	
It remains to justify that $K_1$ is a standard kernel. In view of \eqref{eq:14} we only need to verify the second inequality in \eqref{eq:StanKer}. Observe that $K_1$ is smooth off the diagonal.
	
Notice that for $u < v$ we have
\begin{equation}\label{eq:18}
\begin{split}
	\partial_u K_1(u,v) &=\chi\left( \frac{u}{v}\right)\ind_{\RR_{-}^2}(u,v) \partial_u S_0^{0,1}(u,v) +\frac{1}{v} \chi'\left( \frac{u}{v}\right)\ind_{\RR_{-}^2}(u,v) S_0^{0,1}(u,v),\\
	\partial_v K_1(u,v) &=\chi\left( \frac{u}{v}\right)\ind_{\RR_{-}^2}(u,v) \partial_v S_0^{0,1}(u,v) -\frac{u}{v^2} \chi'\left( \frac{u}{v}\right)\ind_{\RR_{-}^2}(u,v) S_0^{0,1}(u,v).
\end{split}
\end{equation}

We first estimate the partial derivatives of $S_0^{0,1}$.
By \eqref{eq:S0012} and \eqref{eq:Bessel_rec}, we have
\begin{align*}
	\partial_u S_0^{0,1}(u,v) &= \int_0^{1/2} e^{u+v} I_{t+1}(e^{u}) K_{t+1}(e^{v})\,\dd t + \int_0^{1/2} t e^v I_{t}(e^{u}) K_{t+1}(e^v)\, \dd t,\\
	\partial_v S_0^{0,1}(u,v) &= -\int_0^{1/2} e^{2v} I_{t}(e^{u}) K_{t}(e^{v})\,\dd t - \int_0^{1/2} t e^v I_{t}(e^{u}) K_{t+1}(e^v)\, \dd t.
\end{align*}
We emphasize that previously all polynomials in $t$ were disregarded and bounded by constants, whereas here the term $t$ is crucial. Indeed, by applying \eqref{eq:13} with $(k,\ell) = (1,1)$ and $(k,\ell) = (0,1)$, we obtain, for $u<v<0$,
\begin{equation}\label{eq:15}
	\big|\partial_u S_0^{0,1}(u,v)\big| \lesssim e^{2u} + \int_0^{1/2} t e^{-|u-v|t}\,\dd t\simeq e^{2u} + \frac{1}{1+|u-v|^2}.
\end{equation}
Similarly, by applying \eqref{eq:13} with $(k,\ell) = (0,0)$ and $(k,\ell) = (0,1)$, we get
\begin{equation}\label{eq:16}
	\big|\partial_v S_0^{0,1}(u,v)\big| \lesssim e^{v}  + \frac{1}{1+|u-v|^2}.
\end{equation}
Notice that $\chi(u/v)$ vanishes unless $u/v\in (1/4,4)$, and moreover 
\begin{equation*}
	e^u \leq e^{v} = e^{(u-v)/3} e^{(4v-u)/3 }\leq e^{-|u-v|/3},\qquad u<v<0,\ u/v \in (1/4,4).
\end{equation*}
	Consequently, the bounds in \eqref{eq:15} and \eqref{eq:16} simplify to
	\begin{equation}\label{eq:17}
	\big|\partial_u S_0^{0,1}(u,v)\big|,\big|\partial_v S_0^{0,1}(u,v)\big| \lesssim (1+|u-v|)^{-2}
	\end{equation} 
in the region where $u<v<0$ and $\chi(u/v) \neq 0$.
	
Now we consider the second terms in \eqref{eq:18}, where the derivatives fall on $\chi$. Notice that $\chi'(u/v)$ vanishes unless $2\leq u/v\leq 4$ or $1/4\leq u/v\leq 1/2$. Thus, $|u/v|\simeq 1$ and $|v|\simeq |u-v|$. Therefore,
\begin{equation}\label{eq:19}
	\left|\frac{1}{v} \chi'\left( \frac{u}{v}\right)\right|,\left|\frac{u}{v^2} \chi'\left( \frac{u}{v}\right)\right| \lesssim |u-v|^{-1}.
\end{equation}
	
By combining \eqref{eq:17}, \eqref{eq:14}, and \eqref{eq:19} we obtain
\begin{equation*}
	\big|\partial_u K_1(u,v)\big| + \big|\partial_v K_1(u,v)\big| \lesssim |u-v|^{-2},\qquad u\neq v,
\end{equation*}
as claimed.
\end{proof}

\begin{cor}\label{cor:2}
The operators $M_0(\xi)$ are bounded on $L^2(w)$ for any $w \in A_2(\RR)$ and $\xi \in \Rpos$, with a bound depending only on $[w]_{A_2}$.
\end{cor}
\begin{proof}
Again, due to \eqref{eq:Mjn_transl} and the translation-invariance of the $A_2$ characteristic, it is enough to consider $\xi=1$ and prove the weighted bounds for $M_0(1)$.

Let $K_1$ and $K_2$ be the kernels given by Proposition \ref{prop:7}. 
Let $T_1$ and $T_2$ the integral operators with kernels $K_1$ and $K_2$ respectively.
As $S_0^0 = \iker_{M_0(1)}$, we deduce that $M_0(1) = T_1 + T_2$.

Now, the kernel $K_2$ satisfies the pointwise estimate \eqref{eq:12_K2}, which is analogous to the estimate \eqref{eq:12} satisfied by the kernels $S_j^n$ for $(n,j) \neq (0,0)$. Thus, by the same argument as in the proof of Proposition \ref{prop:6}, we deduce the required weighted bounds for $T_2$.

To prove the analogous bounds for $T_1$, we preliminarily observe that $T_1$ is $L^2(\RR)$-bounded, as both $M_0(1)$ and $T_2$ are. Moreover, by Proposition \ref{prop:7}, its integral kernel $K_1$ is a standard kernel. Thus, Theorem \ref{thm:CZ_oper} can be applied to $T_1$ and yields the required weighted bounds.
\end{proof}

\begin{rem}
The proofs of Proposition \ref{prop:7} and Corollary \ref{cor:2} show that one may weaken the assumption on the gradient of the off-diagonal kernel $K$ in Theorem \ref{thm:CZ_oper}: namely, it is enough to assume the second inequality in \eqref{eq:StanKer} only in a conic neighbourhood of the diagonal, i.e., for $C^{-1} \leq u/v \leq C$ for some constant $C>1$. Indeed, much as in the proof of Proposition \ref{prop:7}, one can split $K = K_1+K_2$, where $K_1(u,v) = K(u,v) \chi(u/v)$ is the ``near-diagonal'' part of $K$. The size condition \eqref{eq:StanKer} on $K$ implies that $|K_2(u,v)| \lesssim (|u|+|v|)^{-1}$, so the integral operator $T_2$ corresponding to $K_2$ has all the desired boundedness properties (see \cite[Lemma 5.7]{Ma23}), and in particular it is $L^2$-bounded. On the other hand, the kernel $K_1$ is a standard kernel, and moreover the corresponding singular integral operator $T-T_2$ is $L^2$-bounded, so Theorem \ref{thm:CZ_oper} yields the desired bounds for $T-T_2$.
\end{rem}

The weighted $L^2$-bounds from Proposition \ref{prop:5} and Corollary \ref{cor:2} finally allow us to apply the operator-valued multiplier theorem and complete the proof of our main result.

\begin{proof}[Proof of Theorem \ref{thm:main}]
Recall that $\sqrt{\Lapl}$ has a bounded $\hormander{s}$-calculus on $L^p(N)$ for any $p \in (1,\infty)$ and $s>Q/2$ \cite{Ch,MaMe}. Moreover, Proposition \ref{prop:5} and Corollary \ref{cor:2} prove the validity of the estimate \eqref{eq:23}.
Thus, we apply Theorem \ref{thm:V-VSMT} with $L = \sqrt{\Lapl}$ and $M = M_j$ and in view of Proposition \ref{prop:4} we obtain the $L^p(G)$-boundedness of $\widetilde{\Riesz}_j^\infty$ for $j=0,1$ and all $p\in(1,\infty)$. Hence, Proposition \ref{prop:8} gives the $L^p(G)$-boundedness of the $\Riesz_j^\infty$ for $j=0,1,\ldots,d$ and $p \in (1,\infty)$. Together with Proposition \ref{prop:localparts} and \eqref{eq:Riesz_dec}, this proves that the Riesz transforms $\Riesz_j$ are bounded on $L^p(G)$ for all $p\in(1,\infty)$ and $j=0,\dots,d$.
\end{proof}

\section{Appendix: Bessel functions}\label{S:Appendix}

In this section we recall some standard facts about the Bessel functions, and prove the estimates that are needed in the rest of the paper. In what follows we consider only real arguments and orders.

\subsection{Definitions and basic properties}

The Bessel function of the first kind $J_t$ of order $t\in\RR$ is defined by 
\begin{equation*}
J_t(x)= (x/2)^t \sum_{k=0} ^\infty\frac{(-1)^k  (x/2)^{2k}}{k! \, \Gamma(t+k+1)}, \qquad x>0.
\end{equation*}
Whenever $t=-n$ for some $n\in\NN$, we have $J_{-n}=(-1)^n J_n$ (see for instance \cite[(10.4.1)]{DLMF}).

Below we gather some properties of the Bessel functions of the first kind. Firstly, for $t \in \RR$ the function $J_t$ is smooth on $(0,\infty)$. Moreover, Bessel functions of nonnegative order are bounded (see \cite[(10.14.1)]{DLMF}):
\begin{equation}\label{eq:1}
\big| J_t(x)\big|\leq 1,\qquad x\in\RR,\ t \geq 0.
\end{equation}
Furthermore, we have a recurrence formula for derivatives (see \cite[(10.6.2)]{DLMF}):
\begin{equation}\label{eq:2}
J'_t(x)=J_{t-1}(x)-\frac{t}{x}J_{t}(x) =-J_{t+1}(x)+\frac{t}{x}J_{t+1}(x).
\end{equation}
We shall also make use of the following pointwise bounds:
\begin{equation}\label{eq:4}
\big| J_t(x)\big| \leq Cx^{-1/3},\qquad x>0,\ t>0.
\end{equation}
(see \cite{Lan}), and 
\begin{equation}\label{eq:7}
|J_t(x)|\leq \frac{(x/2)^t}{\Gamma(t+1)},\qquad x>0,\ t>-1/2. 
\end{equation}
(see \cite[(10.14.4)]{DLMF}).

The modified Bessel function of the first kind $I_t$ of order $t \in \RR$, also known as Bessel function of imaginary argument, is defined by
\begin{equation*}
I_t(x)= (x/2)^t \sum_{k=0} ^\infty\frac{(x/2)^{2k}}{k! \, \Gamma(t+k+1)},\qquad x>0.
\end{equation*}
It is easily observed that $I_{-n}=I_n$, $n\in\NN$.

Alternatively, one may use the integral representation for $t>-1/2$ (see \cite[(10.32.2)]{DLMF}):
\begin{equation}\label{eq:11}
I_t(x)=\frac{x^t}{\sqrt{\pi} 2^t \Gamma(t+1/2)} \int_{-1}^1 e^{-xs} (1-s^2)^{t-1/2}\, \dd s.
\end{equation}
For a fixed $t>0$ the function $x\mapsto I_t(x)$ is smooth, positive, and increasing on $(0,
\infty)$, whereas for a fixed $x>0$ the function $t\mapsto I_t(x)$ is decreasing on $(0,\infty)$.

The modified Bessel function of the second kind $K_t$ of order $t \in \RR$ is defined in terms of $I_t$ by
\begin{equation*}
K_t(x) =\frac{\pi}{2} \frac{I_{-t}(x) -I_t(x)}{\sin t\pi},\qquad x>0,
\end{equation*}
where for integer values of $t$ one takes the limit. This expression immediately gives $K_{-t}=K_t$, $t\in\RR$. There is also an integral representation for $K_t$, $t>-1/2$ (see \cite[(10.32.8)]{DLMF}):
\begin{equation}\label{eq:10}
K_t(y) = \frac{\sqrt{\pi}y^t}{2^t \Gamma(t+1/2)} \int_1^\infty e^{-ys} (s^2-1)^{t-1/2}\, \dd s.
\end{equation}
The function $K_t$ is positive and smooth on $(0,\infty)$. Moreover, $x\mapsto K_t(x)$ for a fixed $t$ is decreasing on $(0,\infty)$, and $t\mapsto K_t(x)$ for a fixed $x>0$ is increasing on $(0,\infty)$.

The modified Bessel functions satisfy the following recurrence formulas:
\begin{align}
\begin{split}\label{eq:Bessel_rec}
x\partial_x I_t(x) &= x I_{t-1}(x)- tI_t(x) = xI_{t+1}(x) + tI_t(x),\\
y\partial_y K_t(y) &= -y K_{t-1}(y)- tK_t(y) = -yK_{t+1}(y) + tK_t(y)
\end{split}
\end{align}
(see \cite[(10.27.2)]{DLMF}).

\begin{lm}\label{lm:5}
Let $T>1$ and $x\in(0,1)$. The following estimates hold:
\begin{align*}
	I_t(x)&\simeq_T x^t,\qquad t\in(0,T),\\
	K_t(x)&\simeq_{T} x^{-t},\qquad t\in(T^{-1},T),\\
	K_t(x)&\lesssim x^{t-1},\qquad t\in(0,1/2).
\end{align*}
\end{lm}
\begin{proof}
The estimate for $I_t$ follows immediately from \eqref{eq:11}. For $K_t$, by splitting the integration interval in \eqref{eq:10} onto $(1,2)$ and $(2,\infty)$ we obtain for $t\in(0,T)$ that
\begin{equation*}
	K_t(x) \simeq_T x^t  + x^t \int_2^\infty e^{-xs} s^{2t-1}\,\dd s.
\end{equation*}
If $T=1/2$, then, by bounding $s^{2t-1}\leq 1$ in the integral above, we obtain the desired estimate for $t \in (0,1/2)$. On the other hand, if $t\in(T^{-1},T)$ with $T>1$, then
\begin{equation*}
	x^t\int_2^\infty e^{-xs} s^{2t-1}\,\dd s =   x^{-t} \int_{2x}^\infty e^{-s} s^{2t-1}\,\dd s \simeq_T x^{-t},
\end{equation*}
as required.
\end{proof}

\subsection{Products of modified Bessel functions}

We record here a few formulas for homogeneous derivatives of modified Bessel functions that are repeatedly used in the paper; their proofs are easy applications of the recurrence formulas \eqref{eq:Bessel_rec}.

\begin{lm}
For all $N \in \NN$, $t \in \RR$ and $x,y > 0$,
\begin{equation}\label{eq:higher_der_I_K}
\begin{aligned}
(x\partial_x)^N I_t(x) &= \sum_{k=0}^N C_{N,k}(t) x^{k} I_{t+k}(x), \\
(y\partial_y)^N K_t(y) &= \sum_{k=0}^N C'_{N,k}(t) y^{k} K_{t+k}(y)
\end{aligned}
\end{equation}
for certain polynomials $C_{N,k}$ and $C_{N,k}'$. In addition,
\begin{equation}\label{eq:1st_der_IK}
(x \partial_x - y \partial_y) I_t(x) K_t(y) = x I_{t+1}(x) K_t(y) + y I_t(x) K_{t+1}(y)
\end{equation}
and
\begin{equation}\label{eq:2nd_der_IK}
(x \partial_x + y \partial_y) (x \partial_x - y \partial_y) I_t(x) K_t(y) = (x^2-y^2) I_t(x) K_t(y).
\end{equation}
\end{lm}

We now discuss a couple of estimates for tensor products $I_t(x) K_t(y)$ of modified Bessel functions, as well as certain homogeneous derivatives thereof.

\begin{lm}\label{lm:1}
Let $N\in\NN$. There holds 
\begin{equation}\label{eq:5}
\Big|\big(x\partial_x+y\partial_y\big)^N I_t(x) K_t(y)\Big|\lesssim_N \frac{(t+1) (1+y-x)^{N+1} e^{-(y-x)}}{\sqrt{xy}}
\end{equation}
and 
\begin{equation}\label{eq:6}
\Big|\big(x\partial_x+y\partial_y\big)^N  (x\partial_x - y\partial_y)I_t(x) K_t(y)\Big|\lesssim_N \frac{(t+1)(x+y) (1+y-x)^{N+1} e^{-(y-x)}}{\sqrt{xy}},
\end{equation}
uniformly in $0<x<y<\infty$ and $t>0$.  
\end{lm}
\begin{proof}
Our proof relies on the following formula (see \cite[\S 7.14.2, eq.\ (78)]{EMOT2}):
\begin{equation*}
	I_t(x)K_t(y) = \int_0^\infty J_{2t}\big(2\sqrt{xy}\sinh u \big) e^{-(y-x)\cosh u}\,\dd u,\qquad 0<x<y.
\end{equation*}
After appropriate substitutions we get
\begin{equation*}
	I_t(x)K_t(y)= \int_0^\infty J_{2t}\big(2\sqrt{xy}s\big) e^{-(y-x)\sqrt{1+s^2}}\frac{\dd s}{\sqrt{1+s^2}}= F_t(2\sqrt{xy},y-x),
\end{equation*}
where
\begin{equation}\label{eq:Ftab}
	F_t(a,b) \defeq \frac{1}{a}\int_0^\infty J_{2t}(s) e^{-b\sqrt{1+(s/a)^2}} \frac{\dd s}{\sqrt{1+(s/a)^2}}.
\end{equation}	
	
Notice now that, under the change of variables
\[
a=2\sqrt{xy}, \qquad b=y-x,
\]
we have
\begin{equation*}
x\partial_x+y\partial_y = a\partial_a +b\partial_b, \qquad	x\partial_x-y\partial_y = -\frac{(x+y)}{y-x}b\partial_b.
\end{equation*}
Thus, the estimates \eqref{eq:5} and \eqref{eq:6} reduce to proving that
\begin{equation}\label{eq:Ftab_red}
|(a\partial_a)^{M} (b\partial_b)^{K+\ell} F_t(a,b)| \lesssim_{M,K} \frac{(t+1) b^\ell (1+b)^{M+K+1} e^{-b}}{a},\\
\end{equation}
for $\ell=0,1$, all $M,K \in \NN$ and all $a,b,t>0$.

Now, from \eqref{eq:Ftab} we readily deduce by induction that
\begin{multline*}
	(a\partial_a)^M F_t(a,b)\\
	= \frac{1}{a}\sum_{\substack{n,m\in\NN\\ n+m\leq M}} C_{M,n,m}  \int_0^\infty J_{2t}(s) b^n e^{-b\sqrt{1+(s/a)^2}} \frac{(s/a)^{2n+2m} }{\big(1+(s/a)^2\big)^{m+n/2}} \frac{\dd s}{\sqrt{1+(s/a)^2}}.
\end{multline*}
and
\begin{multline*}
(b\partial_b)^K(a\partial_a)^M F_t(a,b)= 
\frac{1}{a} \sum_{\substack{n,m\in\NN\\ n+m\leq M}} \sum_{\substack{k \in \NN \\ \min(1,K) \leq k\leq K}} C_{M,n,m,K,k} \int_0^\infty J_{2t}(s) \\ 
\times b^{n+k} e^{-b\sqrt{1+(s/a)^2}}
	 \frac{(s/a)^{2n+2m} }{\big(1+(s/a)^2\big)^{m+(n-k)/2}}\frac{\dd s}{\sqrt{1+(s/a)^2}}
\end{multline*}
for suitable real constants $C_{M,n,m}$ and $C_{M,n,m,K,k}$. Consequently,
\begin{equation}\label{eq:est_Ft_Itkm}
|(b\partial_b)^K(a\partial_a)^M F_t(a,b)|
\lesssim_{M,K}  \frac{1}{a}
\sum_{m=0}^M \sum_{k=\min(1,K)}^{M+K} |I_{t,k,m}(a,b)|,
\end{equation}
where
\begin{equation}\label{eq:Itkm_def}
I_{t,k,m}(a,b) = \int_0^\infty J_{2t}(s) \frac{(b \sqrt{1+(s/a)^2})^{k}}{e^{b\sqrt{1+(s/a)^2}}}
	 \frac{(s/a)^{2m}}{\big(1+(s/a)^2\big)^{m}} \frac{\dd s}{\sqrt{1+(s/a)^2}}.
\end{equation}
Thus, the estimates \eqref{eq:Ftab_red} reduce to proving that
\begin{equation}\label{eq:Itkm_red}
|I_{t,k+\ell,m}(a,b)| \lesssim_{k,m} (t+1)b^\ell (1+b)^{k+1} e^{-b}
\end{equation}
for $\ell=0,1$, all $k,m \in \NN$ and all $t,a,b>0$.
	
To prove \eqref{eq:Itkm_red}, we split $I_{t,k,m}$ into two parts $I_{t,k,m}^0$ and $I_{t,k,m}^\infty$, by restricting the integration interval in \eqref{eq:Itkm_def} to $(0,1)$ and $(1,\infty)$, respectively.
Then, by applying \eqref{eq:1} and the estimate
	\begin{equation}\label{eq:3}
	z^n e^{-z} \lesssim_n (1+z_0)^n e^{-z_0}, \qquad z\geq z_0 \geq 0, \ n \in \NN,
	\end{equation}
	we obtain that
	\begin{equation*}
|I^0_{t,k+\ell,m}(a,b)| \lesssim_k b^\ell (1+b)^k e^{-b} \int_0^1	 \frac{(s/a)^{2m}}{\big(1+(s/a)^2\big)^{m+(1-\ell)/2}} \,\dd s \leq b^\ell (1+b)^k e^{-b}. 
\end{equation*}
For the other part we apply the recurrence formula \eqref{eq:2} to obtain that
\[\begin{split}
	&I^\infty_{t,k+\ell,m}(a,b)\\
	&= (2t+1) b^\ell\int_1^\infty \frac{J_{2t+1}(s)}{s} \frac{\big(b\sqrt{1+(s/a)^2}\big)^{k}}{e^{b\sqrt{1+(s/a)^2}}} \frac{(s/a)^{2m} }{\big(1+(s/a)^2\big)^{m+(1-\ell)/2}} \,\dd s\\
	&\quad+b^\ell \int_1^\infty J'_{2t+1}(s) \frac{\big(b\sqrt{1+(s/a)^2}\big)^{k}}{e^{b\sqrt{1+(s/a)^2}}} \frac{ (s/a)^{2m} }{\big(1+(s/a)^2\big)^{m+(1-\ell)/2}} \,\dd s.
\end{split}\]
By using \eqref{eq:4} and \eqref{eq:3} we get that
\begin{equation*}
	\left|\int_1^\infty \frac{J_{2t+1}(s)}{s} \frac{\big(b\sqrt{1+(s/a)^2}\big)^{k}}{e^{b\sqrt{1+(s/a)^2}}} \frac{(s/a)^{2m} }{\big(1+(s/a)^2\big)^{m+(1-\ell)/2}} \,\dd s\right|
	\lesssim_k (1+b)^k e^{-b}.
\end{equation*}
Secondly, integration by parts gives that
\begin{multline*}
			\left|\int_1^\infty J'_{2t+1}(s) \frac{\big(b\sqrt{1+(s/a)^2}\big)^{k}}{e^{b\sqrt{1+(s/a)^2}}}  \frac{ (s/a)^{2m} }{\big(1+(s/a)^2\big)^{m+(1-\ell)/2}} \,\dd s\right|\\
			\leq \left|J_{2t+1}(1)\right| \frac{\big(b\sqrt{1+a^{-2}}\big)^{k}}{e^{b\sqrt{1+a^{-2}}}} \frac{a^{-2m} }{(1+a^{-2})^{m+(1-\ell)/2}}\\
			+ \int_1^\infty \frac{|J_{2t+1}(s)|}{s} \left|s\partial_s \left( \frac{\big(b\sqrt{1+(s/a)^2}\big)^{k}}{e^{b\sqrt{1+(s/a)^2}}} \frac{ (s/a)^{2m} }{\big(1+(s/a)^2\big)^{m+(1-\ell)/2}}\right)\right| \,\dd s.
\end{multline*}
Thus, much as above we obtain that
\[\begin{split}
&\left|\int_1^\infty J'_{2t+1}(s) \frac{\big(b\sqrt{1+(s/a)^2}\big)^{k}}{e^{b\sqrt{1+(s/a)^2}}} \frac{ (s/a)^{2m} }{\big(1+(s/a)^2\big)^{m+(1-\ell)/2}} \,\dd s\right|\\
&\lesssim_k (b+1)^{k} e^{-b}
	+ \sum_{\substack{i,j\in\NN\\ i+j\leq 1}}\int_1^\infty s^{-4/3} b^{k+i}   \frac{ (s/a)^{2(m+i+j)} e^{-b\sqrt{1+(s/a)^2}}}{\big(1+(s/a)^2\big)^{m+j+(i+1-k-\ell)/2}} \,\dd s\\
&\lesssim_k (b+1)^{k+1} e^{-b}.
\end{split}\]
By combining the above bounds we arrive at the estimate \eqref{eq:Itkm_red}, as desired.
\end{proof}

\begin{lm}\label{lm:2}
Let $c>0$, $\varepsilon\in(0,1)$, and $N\in\NN$. Then there holds
\begin{align*}
	\Big| \big( x\partial_x+y\partial_y\big)^N I_t(x)K_t(y)\Big| &\lesssim_{N,c,\varepsilon} (xy)^t e^{-(1-\varepsilon)(y-x)},\\
	\Big| \big( x\partial_x+y\partial_y\big)^N \big(x\partial_x-y\partial_y\big)  I_t(x)K_t(y)\Big| &\lesssim_{N,c,\varepsilon} (x+y)(xy)^t e^{-(1-\varepsilon)(y-x)},
\end{align*}
uniformly in $t\in(0,1)$ and $y>x>0$ such that $y-x\geq c$.
\end{lm}
\begin{proof}
By using the same change of variables and notation as in Lemma \ref{lm:1}, the claimed estimates are reduced to proving that
\[
|(b \partial_b)^{K+\ell} (a\partial_a)^M F_{t}(a,b)| \lesssim_{M,K,c,\varepsilon} a^{2t} b^\ell e^{-(1-\varepsilon) b} 
\]
for $\ell=0,1$ and all $M,K \in \NN$, $t \in (0,1)$, $a>0$ and $b \geq c$. In turn, by \eqref{eq:est_Ft_Itkm}, the latter estimates are reduced to proving that
\[
\frac{1}{a b^\ell} |I_{t,k+\ell,m}(a,b)| \lesssim_{m,k,c,\varepsilon} a^{2t} e^{-(1-\varepsilon) b} 
\]
for $\ell=0,1$ and all $m,k \in \NN$, $t \in (0,1)$, $a>0$ and $b \geq c$.

Now we observe that, by applying the change of variables $s \mapsto as$ and the estimate \eqref{eq:7}, from \eqref{eq:Itkm_def} we obtain that
\[\begin{split}
	\frac{1}{a b^\ell} \big| I_{t,k,m}(a,b)\big|
	&= \left|\int_0^\infty J_{2t}(as) \bigl(b\sqrt{1+s^2}\bigr)^{k} e^{-b\sqrt{1+s^2}} \left(\frac{s^2}{1+s^2} \right)^{m} \frac{\dd s}{(1+s^2)^{(1-\ell)/2}}\right|\\
	&\lesssim a^{2t} \int_0^\infty s^{2t} \bigl(b\sqrt{1+s^2}\bigr)^{k} e^{-b\sqrt{1+s^2}} \,\dd s\\
	&\lesssim_{k,\varepsilon} a^{2t} e^{-(1-\varepsilon)b} \int_0^\infty s^{2t}  e^{- b\varepsilon s/2} \dd s \lesssim_c a^{2t} e^{-(1-\varepsilon)b},
\end{split}\]
where in the last inequality we used that $t \in (0,1)$ and $b \geq c$.
\end{proof}

\end{document}